\documentclass[12pt]{amsart}
\usepackage{amssymb,amsmath,amsthm,amscd,mathrsfs,MnSymbol,enumitem}
\usepackage{fullpage}
\usepackage[justification=centering]{caption}
\usepackage{tablefootnote}
\usepackage{appendix}
\usepackage{hyperref}
\usepackage[msc-links]{amsrefs}
\hypersetup{
  colorlinks   = true,	
  urlcolor     = blue,
  linkcolor    = blue,
  citecolor    = red
}
\theoremstyle{plain}
\newtheorem{prop}{Proposition}[section]
\newtheorem{thm}[prop]{Theorem}
\newtheorem{lem}[prop]{Lemma}
\newtheorem{cor}[prop]{Corollary}
\theoremstyle{definition}
\newtheorem{defn}[prop]{Definition}

\newtheorem{rem}[prop]{Remark}
\numberwithin{equation}{section}
\DeclareMathAlphabet{\mathpzc}{OT1}{pzc}{m}{it}
\newcommand{\beq}{\begin{equation}}
\newcommand{\eeq}{\end{equation}}
\newcommand{\beqn}{\begin{equation*}}
\newcommand{\eeqn}{\end{equation*}}
\newenvironment{psmallmatrix}
  {\left(\begin{smallmatrix}}
  {\end{smallmatrix}\right)}
\begin{document}
\title{K3 Surfaces and Orthogonal Modular Forms}
\begin{abstract}
We determine explicit generators for the ring of modular forms associated with the moduli spaces of K3 surfaces with automorphism group $(\mathbb{Z}/2\mathbb{Z})^2$ and of Picard rank 13 and higher. The K3 surfaces in question carry a canonical Jacobian elliptic fibration and the modular form generators appear as coefficients in the Weierstrass-type equations describing these fibrations.  \end{abstract}
\author{Adrian Clingher}
\address{Department of Mathematics and Statistics, University of Missouri - St. Louis, St. Louis, MO 63121}
\email{clinghera@umsl.edu}
\author{Andreas Malmendier}
\address{Department of Mathematics \& Statistics, Utah State University, Logan, UT 84322}
\email{andreas.malmendier@usu.edu}
\author{Brandon Williams}
\address{Institute of Mathematics, Heidelberg University, 69120 Heidelberg, Germany}
\email{bwilliams@mathi.uni-heidelberg.de}
\maketitle
%
%\section{Motivation}
%
%
\section{Introduction and statement of results}

\subsection{Introduction}
Let $\mathcal{X}$ be a smooth complex projective K3 surface. The intersection form gives the second cohomology group $H^2(\mathcal{X}, \mathbb{Z})$ the structure of an even integral lattice, i.e., a finite-rank free abelian group together with an integer-valued nondegenerate quadratic form. This lattice is the unique (up to isometry) unimodular even lattice of signature $(3, 19)$, which we denote $\Lambda_{K_3} \cong H^{\oplus 3} \oplus E_8(-1)^{\oplus 2}$; in particular it is independent of $\mathcal{X}$. Denote by $\mathrm{NS}(\mathcal{X})$ the N\'eron--Severi lattice, the rank of which is between $1$ and $20$. 
\par In this article, we shall study some specific lattice polarizations on $\mathcal{X}$. Namely, let  $S$ be an even indefinite lattice lattice of signature $(1,\rho_S-1)$ with $1 \leq \rho_S \leq 19$ that admits a canonical primitive lattice embedding 
$S \hookrightarrow \Lambda_{K3}$. As in \cite{MR4635303}*{Section 2B}, fix $h$ to be a very irrational vector in $S\otimes \mathbb{R}$ of positive norm. Then, following \cite{MR4635303}*{Definition 2.6}, as well as \cites{Vinberg2018, MR1420220}, an $S$-polarization on $\mathcal{X}$ is, by definition, a primitive lattice embedding $i \colon S \hookrightarrow \operatorname{NS}(\mathcal{X})$ such that $i(h)$ is big and nef. Two K3 surfaces $(\mathcal{X},i)$ and $(\mathcal{X}',i')$ are said to be isomorphic under $S$-polarization if there exists an analytic isomorphism $\alpha \colon \mathcal{X} \rightarrow \mathcal{X}'$ and a lattice isometry $\beta \in O(S)$ such that $\alpha^* \circ i' = i \circ \beta$, where $\alpha^*$ is the induced morphism at the cohomology level\footnote{Our definition of isomorphic lattice polarizations coincides with the definition of multipolarization used by Vinberg \cite{Vinberg2018}, and is slightly more general than the one used in \cite{MR1420220}*{Sec.~1}.}. Isomorphism classes of lattice polarized K3 surfaces $(\mathcal{X}, i)$ may be fit into a coarse moduli space $\mathscr{M}_S$, which is known to be a quasi-projective variety of dimension $20-\rho_S$; see \cite{MR1420220}.
\par One can study the properties of $\mathscr{M}_S$ from (at least) two points of view. A first such point  of view is the algebraic one. Algebraic families of $S$-polarized K3 surfaces, built via explicit geometric models (quartic normal forms, double sextic surfaces, Weierstrass-type equations), may be used to describe $\mathscr{M}_S$ locally.  Moreover, appropriate choices of parameters for these models may be used to construct algebraic coordinates on  $\mathscr{M}_S$ and such coordinates may lead in some cases to interesting algebraic invariants.  
\par A second point of view is Hodge theoretic. Let $L$ be the orthogonal complement of $S$ in $\Lambda_{K3}$. Then, Hodge structures (periods) for $S$-polarized K3 surfaces are known (see \cite{Vinberg2018}, as well as \cites{MR2030225,MR1420220}) to be classified, up to isomorphism, by the quotient space $ \mathrm{O}^+(L) \backslash \mathscr{D}(L)$, where 
\beq
\label{eqn:ring_automorphic_forms}
  \mathscr{D}(L) \ = \ \Big\{ [z] \in \mathbb{P}\big( L \otimes \mathbb{C} \big) \colon \ \langle z , z \rangle =0 , \ \langle z , \overline{z} \rangle > 0 \Big\}^+ 
 \eeq 
 is a Hermitian symmetric domain of type IV associated with $L$. Here, the index $+$ indicates a choice of connected component and $\mathrm{O}^+(L)  < \mathrm{O}(L)$ denotes the spinor kernel subgroup of integral isometries that preserve the connected component $\mathscr{D}(L)$. 
 %then $\widetilde{\mathrm{O}}(L)$ is the discriminant kernel subgroup of $\mathrm{O}^+(L)$, i.e. the subgroup corresponding to isometries that act trivially on the discriminant group of $L$.  
 %
 \par An appropriate version of the Torelli Theorem (see \cites{Vinberg2018, Vinberg2010}, as well as \cite{MR2030225}) allows one to identify the quasi-projective moduli space $\mathscr{M}_S$ 
 to the quotient $ \mathrm{O}^+(L) \backslash \mathscr{D}(L)$. It should also be noted that 
 $\mathrm{O}^+(L; \mathbb{R})  = \mathrm{SO}(2,20-\rho_S)$ acts transitively on $\mathscr{D}(L)$, which leads to an identification between $\mathscr{D}(L)$ and the bounded symmetric domain  
 \beqn
 \mathrm{SO}(2,20-\rho_S )  / \mathrm{SO}(2)  \times \mathrm{O}(20-\rho_S  )  \;.
\eeqn
The classifying period space $ \mathrm{O}^+(L) \backslash \mathscr{D}(L)$ may then be seen, from this point of view, as a quotient of a bounded symmetric domain by the action of a (modular) discrete group.
\par Cases associated with specific examples for $S$-polarizations of high rank lead to specializations of the modular quotient $ \mathrm{O}^+(L) \backslash \mathscr{D}(L)$ to some classical quotient spaces in algebraic geometry. Such cases were previously studied by the first two authors \cites{MR2369941,MR2935386} \cite{MR1834454}*{Sec.~X.6} , as well as others, and largely motivate the present work. For instance, if $S = H \oplus E_8(-1) \oplus E_8(-1)$, the modular quotient is 2-dimensional and may be identified with the {\it Hilbert modular surface}
\beqn
\left [ \left (  \mathrm{SL}_2(\mathbb{Z}) \times   \mathrm{SL}_2(\mathbb{Z})  \right ) 
\rtimes 
\mathbb{Z} / 2 \mathbb{Z}   \right ] 
\backslash \left (  \mathbb{H} \times \mathbb{H} \right ) ,
\eeqn
where $\mathbb{H}$ is the complex upper half-plane. Similarly interesting, if $S$ is the rank-seventeen lattice $H \oplus E_8(-1) \oplus E_7(-1)$,
the modular quotient in question is isomorphic to the classical {\it Siegel modular three-fold of genus two}, i.e.,
\beqn
\mathrm{Sp}_4(\mathbb{Z})  \backslash   \mathbb{H}_2,
\eeqn
where $\mathbb{H}_2$ is the 3-dimensional complex Siegel modular half-space.    
\par The present paper explores the connection between algebraic descriptions for $S$-polarized K3 surfaces and orthogonal modular forms associated with the domain  $ \mathscr{D}(L)$ and group $\mathrm{O}(L)^+$. This is motivated, as mentioned earlier, by the interesting phenomena seen in the context of high-rank specific example. 
If $S$ is, for instance, given by $H \oplus E_8(-1) \oplus E_8(-1)$, then, as discussed in \cite{MR2369941}, all $S$-polarized K3 surfaces can be given an explicit algebraic construction via {\it Inose forms} --  a special class of quartic equations in $\mathbb{P}^3$. Moreover, the coefficients of the Inose form can be written explicitly in terms of the $\mathrm{SL}_2(\mathbb{Z})$ Eisenstein series $E_4$ and $E_6$, as well as the weight-twelve discriminant $\Delta$. A similar situation arises in the case 
$S = H \oplus E_8(-1) \oplus E_7(-1)$. In this context (see \cite{MR2935386}), the $S$-polarized K3 surfaces can again be described via a generalized version of Inose forms, and the Inose coefficients recover the $\mathrm{Sp}_4(\mathbb{Z})$ Eisenstein series $E_4$ and $E_6$, as well as the Igusa cusp forms $\chi_{10}$ and $\chi_{12}$. These four Siegel modular forms are known to generate the graded algebra of even-weight Siegel modular forms of genus two.    
\par For the present work, we shall restrict our discussion to a specific class of lattice polarizations: those for which the {\it general}\footnote{A $S$-polarized K3 surface $(\mathcal{X},i)$ is called general if $i(S) = \mathrm{NS}(\mathcal{X})$.} $S$-polarized K3 surface $\mathcal{X}$ satisfies ${\rm Aut}(\mathcal{X}) = (\mathbb{Z}/2\mathbb{Z})^2$. 
As we shall explain in Section 3, these polarizing lattices carry a canonical Jacobian elliptic fibration and such a fibration proves to be the essential tool for generalizing the modular form connections noticed above for specific $S$ instances of ranks $18$ and $17$. The Jacobian elliptic fibration in question can be described algebraically via a Weierstrass-type form and the polynomial coefficients appearing in this form (except for possible poles at the Noether--Lefschetz locus), turn out to be modular forms. 
\par For reasons to be explained shortly, we consider only lattices $S$ of rank 13 and higher, Under the above conditions, and via standard results by Nikulin \cites{MR633160, MR556762, MR752938}, Vinberg \cite{MR2429266}, and Kondo \cite{MR1029967}, one obtains that the possible polarizing lattices $S$ reduce to the following  list:
\beq
\label{eqn:lattices}
\begin{split}
 H \oplus E_8(-1)^{\oplus 2},  \ H \oplus E_8(-1)  \oplus E_7(-1),  \\
 H \oplus E_8(-1)  \oplus D_6(-1),   \ H \oplus E_8(-1)  \oplus D_4(-1) \oplus A_1(-1), \\  
 H \oplus E_8(-1)  \oplus A_1(-1)^{\oplus 4}, \ H \oplus D_8(-1)  \oplus D_4(-1), \ H \oplus E_7(-1)  \oplus A_1(-1)^{\oplus 4}.
\end{split}
\eeq
Here, $H$ is the standard rank-two hyperbolic lattice, and the lattices $A_n$, $D_n$, and $E_n$ are the positive definite root lattices associated with the respective root systems.  If one chooses $S$ from Equation~(\ref{eqn:lattices}) with $\rho_S=10+k$ for $3 \le k \le 8$, then one obtains, as orthogonal complement in the K3 lattice, $L = H^{\oplus 2} \oplus A_1(-1)^{\oplus 8-k}$, or  $L = H \oplus H(2) \oplus D_4(-1)$ if $S = H \oplus D_8(-1)  \oplus D_4(-1)$. 

\par In the above cases, there is a good understanding of the rings of modular forms, thanks to works by Woitalla \cite{MR3883325} and Wang and Williams \cite{MR4130466}. Namely, consider the graded ring of modular forms on $\mathscr{D}(L)$ with respect to the group $\mathrm{O}^+(L)$: 
\beq
 \label{eqn:ring_intro}
   M_*(\mathrm{O}^+(L)) = \bigoplus_{m=0}^\infty M_m(\mathrm{O}^+(L)) \ , 
\eeq
as defined for instance in \cite{MR4130466}. This graded ring is known to be finitely generated and, in our cases of interest, one has (see \cites{MR3883325, MR4130466}):
\begin{thm} 
\label{thm_generators}
For lattices $L$ as in the left-side column of the table below , the modular form ring 
$M_*(\mathrm{O}^+(L))$ 
is freely generated by a set of modular forms with weights tabulated in the right-side column:
\begin{center}
\begin{tabular}{l|l}
$L$ & weights\\
\hline
$H \oplus H(2) \oplus D_4(-1)$ & 2, 6, 8, 10, 12, 16, 20\\
$H^{\oplus 2} \oplus A_1(-1)^{\oplus 4}$ & 4, 4, 6, 6, 8, 10, 12\\
$H^{\oplus 2} \oplus A_1(-1)^{\oplus 3}$ & 4, 6, 6, 8, 10, 12\\
$H^{\oplus 2} \oplus A_1(-1)^{\oplus 2}$ & 4, 6, 8, 10, 12\\
$H^{\oplus 2} \oplus A_1(-1)$ & 4, 6, 10, 12\\
$H^{\oplus 2}$ & 4, 6, 12\\
\end{tabular}
\end{center}
\end{thm}
\subsection{Statement of Results}
This paper gives an explicit algebro-geometric interpretation to the modular forms of Theorem \ref{thm_generators}. It follows from \cites{MR4704757, MR4544843, Clingher:2020baq, Clingher:2022} that, for $S$ as in (\ref{eqn:lattices}), the general $S$-polarized K3 surfaces $\mathcal{X}$ carries a unique (up to automorphism) Jacobian elliptic fibration with with Mordell-Weil group isomorphic to $\mathbb{Z}/2\mathbb{Z}$. This fibration has the  canonical van~Geemen-Sarti equation   
\beq
\label{eqn:vgs_a_intro}
  \overline{\mathcal{X}}\colon \quad y^2  = x^3 + A(t) \,   x^2  + B(t) x \,,
\eeq
where $X, Y$ and $t$ are affine coordinates, and $A, B$ are polynomials in $t$ of degree $3$ and $ \leq 8-k$, respectively.  If $S = H \oplus D_8(-1)  \oplus D_4(-1)$, then $\deg(A) = 2$ and $\deg(B) = 5$.  We denote the surface in~(\ref{eqn:vgs_a_intro}) by $\overline{\mathcal{X}}$ to indicate that the components in the reducible fibers not meeting the section have been blown down. In turn, $\mathcal{X}$ is obtained as the minimal resolution of $\overline{\mathcal{X}}$. Moreover, the polynomials $A$ and $B$ can be arranged to have the following form:
\beq
\label{eqn:coeffs_intro}
\begin{array}{lll}
A  =  t^3 + a_4 t + a_6,&
B  =  b_{2(k-2)} t^{8-k} + \dots + b_{12}, &
\text{for $S \neq H \oplus D_8(-1)  \oplus D_4(-1)$,} \\
A =  2\big(c_2 t^2 + c_6 t  + c_{10}\big), &
B =  t^5 + d_8 t^3 + d_{12} t^2 + d_{16} t + d_{20}, &
\text{for $S = H \oplus D_8(-1)  \oplus D_4(-1)$}.
\end{array}
\eeq
In this article, we establish the following theorem:
\begin{thm}
\label{mainth1} 
Consider lattices $S$, as in Equation~(\ref{eqn:lattices}), with $\rho_S \ge 14$. The coefficients $\{a_m, b_n\}$ or $\{c_m, b_n\}$ of Equation~(\ref{eqn:coeffs_intro}) are modular forms with respect to the appropriate group $\mathrm{O}^+(L)$, and are generators for the ring $M_*(\mathrm{O}^+(L))$ in the sense of Theorem~\ref{thm_generators}. 
\end{thm}
A separate situation is given by the remaining rank-$13$ case, i.e., $S=H \oplus E_7(-1)  \oplus A_1(-1)^{\oplus 4}$, with orthogonal complement $L=H^{\oplus 2} \oplus A_1(-1)^{\oplus 5}$. In this situation, a slightly modified version of Theorem \ref{mainth1} holds. The domain $ \mathscr{D}(L) $ contains a special Heegner divisor $C'$ consisting of periods orthogonal to a certain element $r \in L'$ of discriminant form $1/4$ where $L'$ is the dual lattice. Then: 
\begin{thm}
\label{mainth2} 
For $S=H \oplus E_7(-1)  \oplus A_1(-1)^{\oplus 4}$,  the coefficients $\{a_m, b_n\}$ in  Equation~(\ref{eqn:coeffs_intro}) are meromorphic modular forms on $ \mathscr{D}(L) $, of weight 2, 4, 4, 6, 6, 8, 10, 12, with respect to a subgroup 
$ \Gamma \leq \mathrm{O}^+(L) $ of finite co-volume. Moreover, $\{a_m, b_n\}$ form a set of generators for  the ring $M^{!}_*(\Gamma)$  of meromorphic modular forms on $\mathscr{D}(L)$ which are holomorphic away from the Heegner divisor $C'$.
\end{thm}
The proofs for Theorems \ref{mainth1} and \ref{mainth2} are given in Sections 5 and 6. In the case of Theorem \ref{mainth1}, the reasoning follows a recursive argument based on the rank of the polarizing lattice, gradually increasing the complexity of the polynomials in Equations~(\ref{eqn:coeffs_intro}). Due to the known structure of $M(\mathscr{D}(L), \Gamma),$ we have a finite space of possibilities for the coefficients of (\ref{eqn:coeffs_intro}). Moreover, certain expressions in the coefficients can be shown to vanish when the surface belongs to the Noether-Lefschetz locus or when Equation (\ref{eqn:vgs_a_intro}) fails to define a K3 surface entirely. The Noether-Lefschetz locus consists of all the points $(\mathcal{X}, i)$ for which $i\colon S \hookrightarrow \mathrm{NS}(\mathcal{X})$ is a proper inclusion. As we shall see, the Noether-Lefschetz divisors in the period domain correspond exactly to Heegner divisors in the sense of Borcherds \cite{B98} and one can often show that the aforementioned expressions coincide with Borcherds's automorphic products, for which one knows the Fourier coefficients. This, in turn, is enough information to determine the coefficients of (\ref{eqn:coeffs_intro}). In Appendix~\ref{appendix}, we give a list of the coefficients of the van Geemen-Sarti form~(\ref{eqn:vgs_a_intro}) for all lattices $S$ in Equation~(\ref{eqn:lattices}) that we constructed. In turn, our construction of the moduli spaces as open subvarities of suitable weighted projective spaces implies that the moduli spaces of the $S$-polarized K3 surfaces are unirational. This confirms previous observations by M.~Reid \cite{MR605348} and S.~Ma~\cite{MR3275655}. Furthermore, if we set $A\equiv 0$ in Equation~(\ref{eqn:vgs_a_intro}) we obtain a 5-parameter family of K3 surfaces that is parameterized by eight points in $\mathbb{P}^1$, defining the roots of $B$.  This family appeared in work by Kondo in \cite{MR2296434} where he relates the periods of these K3 surfaces to the complex ball uniformization of the moduli space of eight ordered points on the projective line.

\subsection{Previous work}

$S$-polarized K3 surfaces for three cases of $S$ in (\ref{eqn:lattices}) have previously appeared in well-known works of Inose, Matsumoto and other authors. Our work gives another interpretation and proof of their results. Let us briefly explain the geometric motivation behind these examples.
\par A Kummer surface $\mathcal{Y}=\operatorname{Kum}(E_1 \times E_2)$ associated with two non-isogenous elliptic curves $E_1, E_2$ admits several inequivalent elliptic fibrations; see \cites{MR1013073, MR2409557}. In particular, such Kummer surfaces admit an alternate fibration, i.e., an elliptic fibration with section and a Mordell-Weil group that contains a 2-torsion section. Fiber-wise translation by the 2-torsion section gives rise to a canonical symplectic involution, known as van Geemen-Sarti involution. If one factors the Kummer surface $\mathcal{Y}$ by the involution and then resolves the eight occurring singularities, a new K3 surface $\mathcal{X}$ is recovered, related via a rational double-cover map to the Kummer surface. This construction is referred to in the literature as the Nikulin construction.  In turn, each surface $\mathcal{X}$ also admits a canonical van Geemen-Sarti involution. Moreover, if one repeats the Nikulin construction on $\mathcal{X}$, the original Kummer surface is recovered, together with a (generically) two-to-one rational map $\pi\colon \mathcal{X} \to \mathcal{Y}$. One can check that the surfaces $\mathcal{X}$ are polarized by the lattice $H \oplus E_8(-1) \oplus E_8(-1)$; see \cite{MR2369941}. There is also a natural push-forward map $\pi_*$ that restricts to a morphism between the transcendental lattices $\mathrm{T}_\mathcal{X}$ and $\mathrm{T}_\mathcal{Y}$ of $\mathcal{X}$ and $\mathcal{Y}$. In fact, $\pi_*$ induces a Hodge isometry between $\mathrm{T}_\mathcal{X}(2)$ and $\mathrm{T}_\mathcal{Y}$ and we have $\mathrm{T}_\mathcal{X} = \mathrm{T}_{E_1 \times E_2}$. This relation is referred to as a Shioda–Inose structure on $\mathcal{X}$.
\par We shall refer to $\mathcal{X}$ as \emph{Inose K3 surfaces} as they admit a birational model isomorphic to a projective quartic surface introduced by Inose \cite{MR578868}. Shioda and Inose considered these quartic surfaces as an analogue of Weierstrass equations defining elliptic curves. This analogy is particularly fitting since --as we explained above-- the coefficients of the Inose quartic are symmetric expressions of pairs of the modular forms that appear as the coefficients in the Weierstrass equation for an elliptic curve.
\par The entire picture then generalizes to the Picard rank seventeen case when surfaces $mathcal{Y}$ are principally polarized Kummer surfaces: here, the elliptic fibrations on the Jacobian Kummer surfaces were classified in \cite{MR3263663}. The (generalized) Inose K3 surfaces $\mathcal{X}$ are obtained in a similar manner as before from a unique alternate fibration and polarized by the rank seventeen lattice $H \oplus E_8(-1) \oplus E_7(-1)$;  the details may be found in \cites{MR2427457, MR2824841, MR2935386}. The Inose K3 surfaces $\mathcal{X}$ can again be viewed as K3 surfaces admitting Shioda-Inose structures; see \cites{MR728142, MR2279280, MR3087091}.  Now, the coefficients of the defining equations of the Inose K3 surfaces can be expressed in terms of the even generators of the ring of modular forms that were found by Igusa \cites{MR114819, MR141643, MR168805}.
\par Aspects of this construction were generalized for K3 surfaces of lower Picard rank in \cites{MR2824841, MR2254405, MR4015343, CHM19}. Since there are no algebraic Kummer surfaces of Picard rank lower than seventeen, those needed to be replaced by other K3 surfaces; a suitable choice for Picard number sixteen turned out to be the surfaces $\mathcal{Y}$ obtained as double covers of the projective plane branched over the union of six lines. In this way, the rank-seventeen case is recovered by making the six lines tangent to a common conic. The surfaces $\mathcal{Y}$ are polarized by the lattice $H \oplus D_6(-1) \oplus D_4(-1)^{\oplus 2}$. This family was studied by Matsumoto \cites{MR1103969, MR1204828}. The moduli of the family are well understood and are related to Abelian fourfolds of Weil type~\cites{MR3506391, MR1335243}. Again, one can  obtain (generalized) Inose K3 surfaces $\mathcal{X}$ of Picard rank sixteen which are polarized by the lattice $H \oplus E_8(-1) \oplus D_6(-1) \cong H \oplus E_7(-1) \oplus E_7(-1)$; see \cite{CHM19}.  
\par In this article we determine the modular forms that constitute the coefficients in unique equations  for all $S$-polarized K3 surfaces with $S$ given in (\ref{eqn:lattices}), generalizing the explicit relation between modular forms and K3 surfaces with automorphism group $(\mathbb{Z}/2\mathbb{Z})^2$ down to Picard rank 13.
\section*{Acknowledgements}
A.M. acknowledges support from the Simons Foundation through grant no.~202367.  B.W. was supported by Lehrstuhl A f\"ur Mathematik of RWTH Aachen University during most of the work on this paper. Many of the computations in this paper were carried out using the SageMath computer algebra system.
\section{Geometric Aspects: Special Elliptic Fibrations}
\subsection{Preliminaries}
Let us recall some standard facts related to the classification of Jacobian elliptic fibrations on a K3 surfaces. Given a Jacobian elliptic fibration $\pi_\mathcal{X} \colon \mathcal{X} \rightarrow \mathbb{P}^1$ with a section $\sigma_\mathcal{X}$, the cohomology classes of the fiber and section span a rank-two lattice isomorphic to the standard rank-two hyperbolic lattice $H$. One therefore obtains a primitive lattice embedding $ H \hookrightarrow \operatorname{NS}(\mathcal{X})$. The converse of this result also holds: given a primitive lattice embedding $H \hookrightarrow \operatorname{NS}(\mathcal{X})$ whose image contains a quasi-ample class, it is known (see, for instance, \cite{MR2355598}*{Thm.~2.3}) that there exists a Jacobian elliptic fibration on the surface $\mathcal{X}$ whose fiber and section classes span $H$. Moreover,  given $H \hookrightarrow \operatorname{NS}(\mathcal{X})$, one can always compose the embedding with an isometry of $\mathrm{NS}(\mathcal{X})$ in such a way that the image of the resulting embedding contains a quasi-ample class. One therefore has a natural one-to-one correspondence between lattice embeddings $H \hookrightarrow \operatorname{NS}(\mathcal{X})$, up to isometry, and Jacobian elliptic fibrations on $\mathcal{X}$, up to automorphism. 
\par For a given lattice embedding $j \colon H \hookrightarrow \mathrm{NS}(\mathcal{X})$, associated with a Jacobian elliptic fibration, denote by $R(-1)= j(H)^{\perp}$ the orthogonal complement in $\mathrm{NS}(\mathcal{X})$. The lattice $R$  is known as the {\bf frame} of the respective Jacobian elliptic fibration.  Associated with the frame, denote by $R^{\text{root}}$ the sub-lattice spanned by the roots of $R$. Consider also the factor group $\mathpzc{W}= R/ R^{\text{root}}$.  The pair $(R^{\text{root}}, \mathpzc{W})$ is the {\bf type} of the frame, and plays an important role in the classification of Jacobian elliptic fibrations. The lattice $R^{\text{root}}$ is also known as the trivial lattice associated with the elliptic fibration and is a direct sum of ADE  type lattices. $R^{\text{root}}$ provides information about the reducible fibers, while the group $\mathpzc{W}$ is isomorphic to the the Mordell-Weil group of the fibration.
\par Refining the above arguments, one would like to distinguish between non-isomorphic Jacobian elliptic fibrations on $\mathcal{X}$ sharing the same frame, or frame type.  A detailed discussion on this aspect may be found in \cite{Braun:2013aa}. As established by Festi and  Veniani  \cite{FestiVeniani20}, for every frame type $(R^{\text{root}}, \mathpzc{W})$, there are finitely many Jacobian elliptic fibrations of this type supported on $\mathcal{X}$, up to the action of the automorphism group $\mathrm{Aut}(\mathcal{X})$. Moreover, one has a one-to-one correspondence between the possible Jacobian elliptic fibration classes and a certain lattice-theoretic double-coset. The order of this double-coset is known as the {\bf multiplicity} of the frame type.     
\par A full classification of the possible frame pairs $(R^{\text{root}}, \mathpzc{W})$, as well as their corresponding multiplicities, in the case when $\mathrm{Aut}(\mathcal{X})$ is finite, is given in \cite{MR4704757}.
\subsection{Weierstrass equations for Jacobian elliptic fibrations}
Turning to the area of concern for this paper,  we note that various geometric aspects for K3 surfaces $\mathcal{X}$ polarized by lattices $S =  H \oplus R(-1)$, with $R$ a direct sum of positive definite lattices of ADE type, have been studied in previous works of the first two authors \cite{MR4704757}. The general member in this class carries a special Jacobian elliptic fibration with the associated lattice embedding $H \hookrightarrow S$ as the left-side factor in $H \oplus R(-1)$. We refer to this elliptic fibration of frame $R$ as  \textbf{standard}\footnote{We note that multiple non-equivalent splittings $S =  H \oplus R(-1)$, with $R$ of ADE type, may exist. Therefore, the $S$-polarized surface may have multiple non-equivalent standard fibrations.}. These standard fibrations may be described explicitly via affine Weierstrass forms, that is, equations of type $y^2 = x^3 + f(t) x + g(t)$, where $f(t), g(t)$ are polynomials with prescribed properties (degrees, root multiplicities, etc), depending on the choice of $R$.
\par A special sub-case of the above is given by the situation when one has a canonical primitive embedding $ H \oplus N \hookrightarrow S $, where $N$ is the rank-ten Nikulin lattice\footnote{As defined, for instance, in Section 5 of \cite{MR728142}.}. In such a situation, the lattice embedding $H \hookrightarrow S$ corresponding to the left-side factor of $H \oplus N$ determines a second Jacobian elliptic fibration on the $S$-polarized K3 surfaces, which will be referred to as \textbf{alternate}. This fibration has the property that, in addition to the base section, it possesses a second section of order two with respect to the Mordell-Weil group. This type of fibration was studied in detail by van Geemen and Sarti in \cite{MR2274533}. The alternate fibration may again be described via an affine Weierstrass form $y^2 = x^3 + f_a(t) x + g_a(t)$. However, a more effective description is given by a slight modification -- the affine \textbf{van Geemen-Sarti form} $y^2= x^3 + A(t)  x^2 + B(t) x$, where here $A(t), B(t)$ are polynomials of degrees at most 4 and 8, respectively, with prescribed properties. Given a van Geemen-Sarti form, one may recover the Weierstrass form of the fibration via the formulas:
\beqn
\begin{split}
f_a(t) & =  \frac{1}{3} \left( 3 B(t) - A(t)^2 \right), \qquad 
g_a(t) =   \frac{1}{27} A(t) \left( 2 A(t)^2 - 9 B(t) \right), \\[0.2em]
&\Delta_a(t) \ = \ 4 f_a(t)^3 + 27 g_a(t)^2 \ = \ B(t)^2  \left( 4 B(t) - A(t)^2 \right). 
\end{split}
\eeqn
\par For this paper, we shall consider two further restrictions  of the above discussion. First, we shall assume that the polarizing lattice $S$ is of finite automorphism group type\footnote{This condition implies  that $\mathrm{Aut}(\mathcal{X})$  is a finite group, for any K3 surface $\mathcal{X}$ with $\mathrm{NS}(\mathcal{X}) = S$.}, in the sense of Nikulin \cites{MR3165023,MR633160b}.  Second we shall impose a condition on the rank of $S$, namely $13 \leqslant \rho_S \leqslant 18$. Lower rank cases will be discussed in subsequent works. 
\par The following result follows via arguments from \cite{Clingher:2022}. 
\begin{prop}
\label{prop:embedding}
Let $S$ be an even lattice of rank $\rho_S$ and signature $(1, \rho_S-1)$, satisfying the following conditions:
 \begin{enumerate}[label=(\alph*)]
\item $S$ admits a primitive embedding in the K3 lattice $\Lambda_{K3}$,
\item $S$ admits a direct sum decomposition $ H \oplus R(-1)$, with $R$ a positive definite lattice of ADE type,
\item there exists a primitive embedding $ H \oplus N \hookrightarrow S $, 
\item $S$ is of finite automorphism group type, in the sense of Nikulin \cites{MR3165023,MR633160b},
\item $13 \leqslant \rho_S \leqslant 18$. 
\end{enumerate} 
Then, $S$ is isomorphic to one of the lattices of the list below:
\beq
\label{eqn:lattices_body}
\begin{split}
 H \oplus E_8(-1)^{\oplus 2},  \ H \oplus E_8(-1)  \oplus E_7(-1),  \\
 H \oplus E_8(-1)  \oplus D_6(-1),   \ H \oplus E_8(-1)  \oplus D_4(-1) \oplus A_1(-1), \\  
 H \oplus E_8(-1)  \oplus A_1(-1)^{\oplus 4}, \ H \oplus D_8(-1)  \oplus D_4(-1), \ H \oplus E_7(-1)  \oplus A_1(-1)^{\oplus 4}.
\end{split}
\eeq
\end{prop}
\begin{rem} 
\label{rem:lattices}
The following observations regarding the lattices in (\ref{eqn:lattices_body}) are important:
\begin{itemize}
\item [(a)] All lattices involved are of 2-elementary type.\footnote{For a lattice $S$, we denote by $D(S) = S^\vee/S$ the associated discriminant group. $S$ is then called 2-elementary if $D(S)$ is a 2-elementary Abelian group, i.e., $D(S) \cong (\mathbb{Z}/2\mathbb{Z})^\ell$ for $\ell \in \mathbb{N}$.}
\item [(b)] There is a unique instance for each rank, with the exception of rank $14$, where one has two non-isomorphic possibilities. 
\item [(c)] The lattices in (\ref{eqn:lattices_body}) form a nested sequence, in the sense that each lattice has a natural embedding in the one of consecutive higher rank, with the exception of $H \oplus D_8(-1)  \oplus D_4(-1)$ which does not admit an embedding in $ H \oplus E_8(-1)  \oplus D_4(-1) \oplus A_1(-1)$. 
\item [(d)] The embeddings of (c) are unique, up to an isometry of the larger lattice. 
\end{itemize} 
\end{rem} 
We also note that, via classical results by Nikulin~\cites{MR633160, MR556762, MR752938}, Vinberg~\cite{MR2429266}, and Kondo~\cite{MR1029967}, one obtains:
\begin{prop}
\label{prop:finite_autos}
 Let $\mathcal{X}$ be a K3 surface with Picard rank $13 \leqslant \rho_X \leqslant 18$. Then the automorphism group $\mathrm{Aut}(\mathcal{X}) $ is isomorphic to $(\mathbb{Z}/2\mathbb{Z})^2$  if and only if the Neron-Severi lattice $\mathrm{NS}(\mathcal{X})$ is isomorphic to one of the lattices in~(\ref{eqn:lattices_body}). 
 \end{prop}
The following observation, which follows from \cite{Clingher:2022}*{Theorem 2.3}, is important for the considerations of this paper:
\begin{prop}
\label{prop:frame}
Let $S$ be one of the lattices in~(\ref{eqn:lattices_body}) and assume that $\mathcal{X}$ is a general $S$-polarized K3 surface. Then, there is only one frame on $\mathcal{X}$ with $\mathpzc{W}=\mathbb{Z}/2\mathbb{Z}$, and the multiplicity of this frame is always 1.
\end{prop}

\begin{rem} 
\label{rem:lattices1}
Re-framing Proposition \ref{prop:frame} within the standard/alternate terminology established above, we conclude that, for a general  $S$-polarized K3 surface $\mathcal{X}$, with $S$ in (\ref{eqn:lattices_body}),  while multiple non-isomorphic elliptic Jacobian fibrations with standard frame $R_{\text{std}}$ may occur, one always has a unique (canonical) alternate fibration with root lattice of the frame $R^{\text{root}}_{\text{alt}}$ and $\mathpzc{W}_{\text{alt}}=\mathbb{Z}/2\mathbb{Z}$. All possible standard/alternate frames\tablefootnote{Multiplicities greater than 1 may occur for some standard frames, as detailed in \cite{MR4704757}.} are detailed in in Table~\ref{tab:lattices}.
\end{rem} 
\begin{table}[!th]
\begin{center}
\begin{tabular}{c|c|cc}
$\rho_S$ &  $S \simeq H \oplus R_{\text{std}}(-1)$ & $R^{\text{root}}_{\text{alt}}$ &  $\mathpzc{W}_{\text{alt}}$ 	\\
\hline
\hline
$18$ 	& $H \oplus E_8(-1)  \oplus E_8(-1)$ & $D_{16}$    & $\mathbb{Z}/2\mathbb{Z}$\\
\hline
$17$ 	& $H \oplus E_8(-1)  \oplus E_7(-1)$ 			   & $D_{14} + A_1$ 	& $\mathbb{Z}/2\mathbb{Z}$\\
\hline
$16$	& $H \oplus E_7(-1) \oplus E_7(-1)$                & $D_{12} + 2 A_1$ 	&$\mathbb{Z}/2\mathbb{Z}$\\
        & $\cong H \oplus E_8(-1)  \oplus D_6(-1)$ &&\\		
        & $\cong H \oplus D_{14}(-1)$ &&\\		
\hline        
$15$	& $H \oplus E_8(-1)  \oplus D_4(-1) \oplus A_1(-1)$ & $D_{10} + 3 A_1$	&$\mathbb{Z}/2\mathbb{Z}$ \\
        & $\cong H \oplus E_7(-1) \oplus D_6(-1)$ &&\\	    
        & $\cong H \oplus D_{12}(-1) \oplus A_1(-1)$ &&\\		
\hline
$14$    & $H\oplus E_8(-1) \oplus A_1(-1)^{\oplus 4}$ 		& $D_8 + 4 A_1$ 	& $\mathbb{Z}/2\mathbb{Z}$\\			
		& $\cong H \oplus E_7(-1) \oplus D_4(-1) \oplus A_1(-1)$ &&\\
        & $\cong H \oplus D_{10}(-1) \oplus A_1(-1)^{\oplus 2}$ &&\\
		& $\cong H \oplus D_{6}(-1)^{\oplus 2}$ &&\\
\hline
$14$ 	& $H\oplus D_8(-1) \oplus D_4(-1)$ 		             & $E_7 + 5 A_1$ 	& $\mathbb{Z}/2\mathbb{Z}$\\
\hline
$13$ 	& $H\oplus E_7(-1) \oplus A_1(-1)^{\oplus 4}$ 	     & $D_6 + 5 A_1$ 	& $\mathbb{Z}/2\mathbb{Z}$\\
        & $\cong H \oplus D_8(-1)  \oplus A_1(-1)^{\oplus 3}$&&\\
		& $\cong H \oplus D_6(-1)  \oplus D_4(-1) \oplus A_1(-1)$ &&\\
\hline  
\end{tabular}
\end{center}
\captionsetup{justification=centering}
\caption{Lattices in Remark~\ref{rem:lattices}}
\label{tab:lattices}
\end{table}
\subsection{Van Geemen-Sarti forms}
The canonical alternate fibrations $\pi_\mathcal{X}\colon  \mathcal{X} \rightarrow \mathbb{P}^1$ of Proposition \ref{prop:frame} may be described explicitly. We shall give these descriptions here and construct the corresponding van Geemen-Sarti form
\beq
\label{eqn:vgs}
  \overline{\mathcal{X}}\colon \quad Y^2 Z  = X^3 + A(u, v) \, X^2 Z +  B(u, v) \, X Z^2,
 \eeq
with elliptic fibers in $\mathbb{P}^2 = \mathbb{P}(X, Y, Z)$ varying over $\mathbb{P}^1=\mathbb{P}(u, v)$, for the subset of lattices with Picard rank 13 and higher:
\begin{prop}
\label{prop:weierstrass}
For $S$ in~(\ref{eqn:lattices}) with $\rho_S=10+k$ for $4 \le k \le 8$ and $L = S^{\perp}_{\Lambda_{K3}}$, the defining polynomials in Equation~(\ref{eqn:vgs}) can be arranged to have the form 
\beq
\label{eqn:coeffs_a}
 A = \big(u^3 + a_4 u v^2 + a_6 v^3\Big) v,  \qquad  B 
 = \Big(b_{2(k-2)} u^{8-k} + \dots + b_{12} v^{8-k}\big) v^k
 = \sum_{l=0}^{8-k} b_{2(6-l)} u^{l} v^{8-l}
\eeq  
for $L = H^{\oplus 2} \oplus A_1(-1)^{\oplus 8-k}$, and
\beq
\label{eqn:coeffs_b}
 A = 2\Big(c_2 u^2 + c_6 u v  + c_{10} v^2\Big)v^2,  \qquad  B = \Big(u^5 + d_8 u^3 v^2 + d_{12} u^2v^3 + d_{16} u v^4 + d_{20} v^5\big)v^3
\eeq  
for $L = H \oplus H(2) \oplus D_4(-1)$. The coefficients in Equation~(\ref{eqn:coeffs_a}) satisfy $( b_{2(k-2)}, \dots ,   b_{12} )  \neq  0$ for $k=5, 6, 7, 8$ and
\beq
\label{eqn:locus_sing_a}
 \Big(\,  
 2 a_4 b_4 + b_8, \ 
 a_4 b_6 + 3 b_{10}, \ 
 a_4 b_6 - 6 a_6 b_4, \ 
 a_4 b_8 + 18 b_{12},  \ 
 8 b_4 b_8 - 3 b_6^2, \
 4 a_4 b_8 - 9 a_6 b_6 \,  \Big)  \  \neq \  0 
\eeq
for $k=4$, and the coefficients in Equation~(\ref{eqn:coeffs_b}) satisfy 
\beq
\label{eqn:locus_sing_b}
\begin{split}
    \Big( \, 
    c_2 d_8 + 10 c_{10}, \
    4 c_2 c_{10} - c_6^2, \
    c_2 d_{12} - c_6 d_8, \
    3 d_8^2 + 20 d_{16},  \ & \\
    8 c_2 d_{16} - 3 c_6 d_{12}, \
    d_8 d_{12} + 50 d_{20},  \
    8 d_8 d_{16} -3 d_{12}^2    
    \, \Big)  &  \ \neq  \  0.
\end{split}    
\eeq
\end{prop}
\begin{proof}
Proposition~\ref{prop:frame} shows that a general $S$-polarized K3 surface $\mathcal{X}$ has a Jacobian elliptic fibration $\pi_\mathcal{X}\colon  \mathcal{X} \rightarrow \mathbb{P}^1$ with $\operatorname{MW}(\mathcal{X}, \pi_\mathcal{X})\cong \mathbb{Z}/2\mathbb{Z}$. This yields Equation~(\ref{eqn:vgs}), if we move the zero-section to $[X: Y: Z] = [0:1:0]$ and the 2-torsion section to $[X: Y: Z] = [0:0:1]$.
\par In the first case, it was proved in \cite{Clingher:2022} that the elliptic fibration has the singular fibers $I^*_{2(k-2)} + (8-k) I_2+6 I_1$. The van Geemen-Sarti form~(\ref{eqn:vgs}) with a singular fiber of type $I^*_{2(k-2)}$ at $v=0$ has $A=(a_0 u^3 + \dots + a_6 v^3)v$ and $B=(b_{2(k-2)} u^{8-k} + \dots + b_{12} v^{8-k})v^k$. It follows that $B \neq 0$ since otherwise the surface has non-isolated singularities. Moreover, in order for the elliptic surface to have at worst a rational double point singularity at $v=0$, we must have $a_0 \neq 0$. This can be seen by converting Equation~(\ref{eqn:vgs}) into its standard Weierstrass form
\beq
\label{eqn:vgs_a_prelim_weq}
 y^2 z = x^3 - \frac{1}{3} \Big( A(u, v)^2 - 3  B(u, v) \Big) \, xz^2 + \frac{1}{27}  A(u, v) \Big(2 A(u, v)^2 - 9  B(u, v)\Big) \, z^3.
\eeq 
For $a_0=0$ we obtain a $(4,6, 12)$-point at $v=0$ whence a non-minimal Weierstrass equation. Thus, we must have $a_0 \neq 0$. For $a_0 \neq 0$ the coefficient $a_0$ can be rescaled to equal one, and $a_2$ can be subsequently eliminated by a shift in $u$. We then obtain Equation~(\ref{eqn:vgs}) with defining polynomials given by Equation~(\ref{eqn:coeffs_a}). Moreover, any $(4,6, 12)$-point with $v \neq 0$ requires $A = (u+\beta v)^2 (u -2 \beta v) v$ and $B = b_4 ( u + \beta v)^4 v^4$ for some $\beta \in\mathbb{C}$ and $b_4 \neq 0$. This is only possible for $k\le 4$, and Equation~(\ref{eqn:locus_sing_a}) follows by computing a Gr\"obner basis for the intersection of the corresponding elimination ideal and the ideal generated by $b_4 = \dots = b_{12}=0$. The second case is analogous: the singular fibers are $III^*+ 5 I_2+5 I_1$; see \cite{Clingher:2024}. We then moved the fiber of type $III^*$ to $v=0$. Here, a $(4,6, 12)$-point with $v\neq 0$ requires $A = 2c_2 (u+\beta v)^2v^2$ and $B = (u - 4 \beta v) ( u + \beta v)^4v^3$. The left side of Equation~(\ref{eqn:locus_sing_b}) is obtained as Gr\"obner basis of the corresponding elimination ideal.
\end{proof}
The following is immediate:
\begin{cor}
\label{cor:uniqueness}
The coefficients of the alternate fibrations determined in Proposition~\ref{prop:weierstrass} are unique, up to a rescaling $(a_m, b_n) \mapsto (\lambda^m a_m, \lambda^n b_n)$ and $(c_m, d_n) \mapsto (\lambda^m c_m, \lambda^n d_n)$, respectively, with $\lambda \in \mathbb{C}^\times$. 
\end{cor}
\begin{proof}
Proposition~\ref{prop:frame} proves that the Jacobian elliptic fibration is unique. Among the singular fibers there is a unique fiber of type $I^*_{2(k-2)}$ or $III^*$, respectively, that has been moved to $v=0$. After the subsequent shift in $u$, the only remaining holomorphic coordinate change which preserves the van Geemen-Sarti form~(\ref{eqn:vgs}) and the leading coefficient of $A$ in the first case or $B$ in the second case, in the affine chart $Z=v=1$, is the rescaling $(u, X, Y) \mapsto (\lambda^2 u, \lambda^6 X, \lambda^9 Y)$ or $(u, X, Y) \mapsto (\lambda^4 u, \lambda^{10} X, \lambda^{15} Y)$ in the second case. One then checks that this leads to the stated rescalings of the coefficients.
\end{proof}
\begin{cor}
\label{cor:recursiveI}
Setting $b_{2(k-2)}=0$ in the van Geemen-Sarti form~(\ref{eqn:vgs}) with coefficients in~(\ref{eqn:coeffs_a}), the alternate fibration for the polarizing lattice $S'$ with $L' = H^{\oplus 2} \oplus A_1(-1)^{\oplus 7-k}$ is attained. Moreover, this is the only way to obtain the alternate equation for $S'$.
\end{cor}
\begin{proof}
The proof is by a recursive argument on the rank of the polarizing lattice. It follows from the work in \cite{MR890925} that from a fibration with singular fibers $I^*_{2(k-2)} + (8-k) I_2+6 I_1$ a fibration with singular fibers $I^*_{2(k-1)} + (7-k) I_2+6 I_1$ can only be obtained by the confluence of the singular fiber of type $I^*_{2(k-2)}$ and $I_2$. This implies that the degree of $B$ as a polynomial in $u$ must drop by one.
\end{proof}
In the remaining case in Equation~(\ref{eqn:lattices}) we have the following:
\begin{prop}
\label{prop:weierstrass_13}
For $S=H \oplus E_7(-1)  \oplus A_1(-1)^{\oplus 4}$ and $L = S^{\perp}_{\Lambda_{K3}} =H^{\oplus 2} \oplus A_1(-1)^{\oplus 5}$, the defining polynomials in Equation~(\ref{eqn:vgs}) can be arranged to have the form 
\beq
\label{eqn:coeffs_13}
\begin{split}
 A = \Big(u^3 + a_4 u v^2 + a_6 v^3\Big) v,  \quad  B 
 = \Big( b_{2} u^5 + b_4 u^4 v + b_6 u^3 v^2 + b_8 u^2 v^3 + b _{10} u v^4+ b_{12} v^{5} \big) v^3,
\end{split}
\eeq  
or
\beq
\label{eqn:coeffs_13_II}
\begin{split}
 A = 2 \Big( c_{-2} u^3 + c_2 u^2 v + c_6 u v^2 + c_{10} v^3\Big)v , \quad  B  = u^5 + d_8 u^3v^2 + d_{12} u^2v^3 + d_{16} uv^4 + d_{20} v^5.
\end{split}
\eeq  
The coefficients in Equation~(\ref{eqn:coeffs_13}) satisfy 
\beq
\label{eqn:locus_sing_13_a}
\begin{split}
  \Big(\, 
 2 a_4 b_4 - 10 a_6 b_2 + b_8, \ 
 a_4 b_6 - 2 a_6 b_4 + 2 b_{10},  \ 
 10 a_4^2 b_2 - 3 a_4 b_6 + 18 a_6 b_4, \ 
 a_4 b_8 - a_6 b_6 + 10 b_{12},  \\ 
 8 a_4^2 b_4 - 30 a_4 a_6 b_2 + 9 a_6 b_6, \ 
 20 b_2 b_{10} - 8 b_4 b_8 + 3 b_6^2, \ 
 4 a_4^3 b_2 + 27 a_6^2 b_2 
 \, \Big)  \ \neq  \ 0 ,
\end{split}
\eeq
and the coefficients in Equation~(\ref{eqn:coeffs_13_II}) satisfy
\beq
\label{eqn:locus_sing_13_b}
\begin{split}
    \Big( \, 
    c_{-2} d_{12} - c_2 d_8 - 10 c_{10}, \ 
    3 c_{-2}^2 d_{16} - 3 c_{-2} c_2 d_{12} + 2 c_2^2 d_8 + 5 c_6^2, \ 
    2 c_{-2} d_{16} - c_2 d_{12} + c_6 d_8, \\ 
    3 d_8^2 + 20 d_{16},  \
    45 c_{-2} d_{20} - 8 c_2 d_{16} + 3 c_6 d_{12}, \
    d_8 d_{12} + 50 d_{20}, \ 
    8 d_8 d_{16} -3 d_{12}^2  
    \,\Big)  \  \neq  \ 0.
\end{split}    
\eeq
\end{prop}
\begin{proof}
The proof is analogous to the proof of Proposition~\ref{prop:weierstrass}, except now we have set $k=3$. Equation~(\ref{eqn:vgs_a_prelim_weq}) shows that we have a $(4,6, 12)$-point at $v=0$ if and only if $(a_0, b_2)=(0, 0)$. For $a_0 \neq 0$ we rescale so that we can assume $a_0=1$ and eliminate $a_2$ by a suitable shift in $u$ to obtain the coefficients in Equation~(\ref{eqn:coeffs_13}). Equation~(\ref{eqn:locus_sing_13_a}) follows as in the proof of Proposition~\ref{prop:weierstrass} since a $(4,6, 12)$-point with $v\neq 0$ requires $A = (u+\beta v)^2 (u -2 \beta v)v$ and $B = (b_2 u + \alpha v) ( u + \beta v)^4v^3$ for some $\alpha, \beta \in\mathbb{C}$. The case $b_2 \neq 0$ works similarly and yields the coefficients given by Equation~(\ref{eqn:coeffs_13_II}). Here, a $(4,6, 12)$-point with $v\neq 0$ requires $A = 2c_{-2}(u+\alpha v)(u+\beta v)^2 v$ and $B = (u - 4 \beta v) ( u + \beta v)^4v^3$ for some $\alpha, \beta \in\mathbb{C}$.
\end{proof}
We have the following:
\begin{lem}
\label{lem:transfo_coeffs}
For $c_{-2} \neq 0$ the coefficients in~(\ref{eqn:coeffs_13_II}) can be transformed into coefficients in~(\ref{eqn:coeffs_13}), with $b_2 = 1/(2c_{-2})$ and
\beq
\begin{split}
 a_4 = 4 c_{-2} c_6 - \frac{4}{3} c_2^2, \quad b_4 = - \frac{5}{3} \frac{c_2}{c_{-2}}, \quad a_6 = 8 c_{-2}^2 c_{10} - \frac{8}{3} c_{-2} c_2 c_6 + \frac{16}{27} c_2^3, \quad b_6 = \frac{20}{9} \frac{c_2^2}{c_{-2}} + 2 c_{-2} d_8,\\
 b_8 = - \frac{40}{27} \frac{c_2^3}{c_{-2}} - 4 c_{-2} c_2 d_8 + 4 c_{-2}^2 d_{12},\quad
 b_{10} = \frac{40}{81} \frac{c_2^4}{c_{-2}} + \frac{8}{3} c_{-2} c_2^2 d_8 - \frac{16}{3} c_{-2}^2 c_2 d_{12} + 8 c_{-2}^3 d_{16}, \\
 b_{12} = - \frac{16}{243} \frac{c_2^5}{c_{-2}} - \frac{16}{27} c_{-2} c_2^3 d_8 + \frac{16}{9} c_{-2}^2 c_2^2 d_{12} - \frac{16}{3} c_{-2}^3c_2 d_{16} + 16 c_{-2}^4 d_{20}.
\end{split}
\eeq
A similar statement holds for $b_2 \neq 0$. 
\end{lem}
\begin{cor}
\label{cor:uniqueness13}
The coefficients of the alternate fibration determined in Proposition~\ref{prop:weierstrass_13} are unique, up to changing coefficients according to Lemma~\ref{lem:transfo_coeffs} or a rescaling $(a_m, b_n) \mapsto (\lambda^m a_m, \lambda^n b_n)$ and $(c_m, d_n) \mapsto (\lambda^m c_m, \lambda^n d_n)$, respectively, with $\lambda \in \mathbb{C}^\times$. 
\end{cor}
\begin{cor}
\label{cor:recursiveII}
Setting $b_2 =0$ (or $c_{-2}=0$) in the van Geemen-Sarti form~(\ref{eqn:vgs}) with coefficients in~(\ref{eqn:coeffs_13}) (or coefficients in~(\ref{eqn:coeffs_13_II})), the alternate fibration for the polarizing lattice $S'=H \oplus E_8(-1)  \oplus A_1(-1)^{\oplus 4}$ (or $S'=H \oplus D_8(-1)  \oplus D_4(-1)$, respectively) is attained. Moreover, this is the only way to obtain the alternate equation for $S'$.
\end{cor}%
\begin{rem}
In the affine coordinates $Z=v=1$, $X=x, Y=y$, and $u=t$, Equation~(\ref{eqn:vgs}) becomes Equation~(\ref{eqn:vgs_a_intro}) from the introduction. Moreover, a canonical representative for a nowhere vanishing holomorphic two-form $\omega_\mathcal{X}$ on $\mathcal{X}$ is determined on $\overline{\mathcal{X}}$ in the affine coordinate chart by $dt \wedge dx/y$. The minimal resolution of the van Geemen-Sarti equation yields an $S$-polarized K3 surface $(\mathcal{X}, i)$, and any isometry $\phi\colon H^2(\mathcal{X}, \mathbb{Z}) \to \Lambda_{K3}$ of lattices, with $i = \phi^{-1}\vert_S$, then assigns to $\omega_\mathcal{X}$ a point $z_0 = \phi(\omega_\mathcal{X}) \in  \mathbb{P}\big( L \otimes \mathbb{C} \big)$ with $L = S^{\perp}_{\Lambda_{K3}}$. The pseudo-orthogonal group $\mathrm{O}(L)$ contains a subgroup $\mathrm{O}(L)^+ < \mathrm{O}(L)$ of index 2 preserving the component $\mathscr{D}(L)$ with $z_0 \in \mathscr{D}(L)$. The marking $\phi$ is determined up to a left multiplication by an element of $\mathrm{O}^+(\Lambda_{K3})$ leaving $z_0$ and $L$ invariant, and $\Gamma=\mathrm{O}^+(L, z_0)$ is precisely the arithmetic subgroup of finite index in $\mathrm{O}(L)^+$ formed by the restrictions to $L$ of such operators. 
\end{rem}
\begin{rem}
Propositions~\ref{prop:weierstrass} and~\ref{prop:weierstrass_13} provide a construction for the moduli spaces of $S$-polarized K3 surfaces with $S$ in Equation~(\ref{eqn:lattices}) as open subvarities in the weighted projective spaces whose coordinate rings are generated by the coefficients of $A, B$. The complement of the discriminant locus within these weighted projective spaces is given by $B \not \equiv 0$ for $k>4$ and Equation~(\ref{eqn:locus_sing_a}) (resp.~(\ref{eqn:locus_sing_13_a})) for $k=4$ and Equation~(\ref{eqn:locus_sing_b}) (resp.~(\ref{eqn:locus_sing_13_b})) for $k=3$. Our construction then implies immediately that the moduli spaces of the $S$-polarized K3 surfaces are unirational and confirms previous observations by M.~Reid \cite{MR605348} and S.~Ma~\cite{MR3275655}.
\end{rem}
\begin{rem}
If we set $c_2 = c_6 = c_{10} = 0$ in Equation~(\ref{eqn:vgs}) with defining polynomials in Equation~(\ref{eqn:coeffs_b}), we obtain the sub-family of the $H\oplus D_8(-1) \oplus D_4(-1)$-polarized family with singular fibers $III^* + 5 III$. This family also appeared in \cite{MR2296434}. In his work Kondo considered the complex ball uniformization of the moduli space of eight ordered points on the projective line by using the theory of periods of certain K3 surfaces, also building on earlier work in \cite{MR1780433}. In terms of Equation~(\ref{eqn:vgs}), the corresponding K3 surfaces are given by $A\equiv 0$ and the aforementioned eight points in $\mathbb{P}^1$ define the roots of $B$.
\end{rem}
\section{Orthogonal modular forms}

\subsection{Modular forms on orthogonal groups}

Let $L$ be an even integral lattice of signature $(2, n)$ and define $L_{\mathbb{R}} = L \otimes \mathbb{R}$ and $L_{\mathbb{C}} = L \otimes \mathbb{C}$. \\

The Hermitian symmetric domain $\mathcal{D}$ for $\mathrm{O}(L)$ is the space $$\mathcal{D} = \Big\{ [z] \in \mathbb{P}(L_{\mathbb{C}}): \; \langle z, z \rangle = 0, \; \langle z, \overline{z} \rangle > 0, \; \{\mathrm{re}[z], \mathrm{im}[z]\} \; \text{oriented} \Big\}.$$
In other words, if $z = x+iy$ then $[z] \in \mathcal{D}$ if and only if, up to scalar, $(x, y)$ is an oriented orthonormal basis of a positive-definite plane in $L_{\mathbb{R}}$. The affine cone over $\mathcal{D}$ is the set $$\mathcal{A} = \Big\{ z \in L_{\mathbb{C}}: \; [z] \in \mathcal{D} \Big\}$$ of points that span a line in $\mathcal{D}$. \\

Let $\mathrm{O}(L)$ be the orthogonal group of $L$. There is an index two subgroup $\mathrm{O}^+(L)$ (the spinor kernel) of $\mathrm{O}(L)$ defined by the fact that it preserves the set $\mathcal{D}$ under multiplication. It is also characterized by the following fact: for $r \in L_{\mathbb{R}}$, the reflection $$\sigma_r \colon v \mapsto v - \frac{\langle v, r \rangle}{\langle r, r \rangle} r$$ belongs to $\mathrm{O}^+(L_{\mathbb{R}})$ if and only if $r^2 > 0$.\\

Let $\Gamma \le \mathrm{O}^+(L)$ be a finite-index subgroup. A meromorphic function $f\colon \mathcal{A} \rightarrow \mathbb{C}$ is called a \textbf{meromorphic modular form} of weight $k$ and level $\Gamma$ if
\begin{enumerate}
\item $f(\gamma z) = f(z)$ for all $\gamma \in \Gamma$,
\item $f(\lambda z) = \lambda^{-k} f(z)$ for $\lambda \in \mathbb{C}^{\times}$.
\end{enumerate}
The most important such $\Gamma$ (besides $\mathrm{O}^+(L)$ itself) is the discriminant kernel $$\widetilde{\mathrm{O}}(L) = \Big \{ \gamma \in \mathrm{O}^+(L): \; \gamma x - x \in L \; \text{for all} \; x \in L' \Big\}.$$

\subsection{Jacobi forms and lifts}

Throughout this section, we assume that we have fixed a splitting $L = H \oplus H \oplus K(-1)$ with a positive-definite lattice $K$.
The dual lattice is labeled $$K' = \{x \in K \otimes \mathbb{Q}: \; \langle x, y \rangle \in \mathbb{Z} \; \text{for all} \; y \in K\}.$$

\begin{defn} A \textbf{Jacobi form} of weight $k$ and index $K$ is a holomorphic function $$\phi \colon \mathbb{H} \times K_{\mathbb{C}} \longrightarrow \mathbb{C}$$ that satisfies the functional equations:
$$\phi\left( \frac{a \tau + b}{c \tau + d}, \frac{z}{c \tau + d} \right) = (c \tau + d)^k \exp \Big( \pi i \frac{c \cdot \langle z, z \rangle}{c \tau + d} \Big) \cdot \phi(\tau, z)$$ and $$\phi \Big( \tau, z + \lambda \tau + \mu \Big) = \exp \Big( -\pi i \langle \lambda, \lambda \rangle \tau - 2\pi i \langle \lambda, z \rangle \Big) \cdot \phi(\tau, z)$$ for any $\begin{psmallmatrix} a & b \\ c & d \end{psmallmatrix} \in \mathrm{SL}_2(\mathbb{Z})$ and $\lambda, \mu \in K$.
\end{defn}

Jacobi forms are represented by their Fourier series $$\phi(\tau, z) = \sum_{n = -\infty}^{\infty} \sum_{r \in K'} c(n, r) q^n \zeta^r,$$ where $q^n := e^{2\pi i n \tau}$ and $\zeta^r := e^{2\pi i \langle z, r \rangle}$. The Jacobi form $\phi$ is called
\begin{itemize}
\item \textbf{weakly holomorphic} if $c(n, r) = 0$ for sufficiently small $n$ and every $r$;
\item \textbf{weak} if $c(n, r) = 0$ for every $n < 0$ and every $r$;
\item \textbf{holomorphic} if $c(n, r) = 0$ whenever $n - \langle r, r \rangle / 2 < 0$; 
\item a \textbf{cusp form} if $c(n, r) = 0$ whenever $n - \langle r, r \rangle / 2 \le 0$.
\end{itemize}
The value $n - \langle r, r \rangle / 2$ is also called the hyperbolic norm of the tuple $(n, r)$.

\noindent
$G = \mathrm{Aut}(K'/K)$ operates on Jacobi forms of index $K$ via  $\phi^g(\tau, z) := \phi(\tau, gz).$ \\

If $f\colon \mathcal{A} \rightarrow \mathbb{C}$ is a modular form on the discriminant kernel $\widetilde{\mathrm{O}}(L)$ then $f$ can be expressed as a Fourier-Jacobi series as follows. Write elements of $L$ as tuples in the form $(a, b, x, c, d)$, with $a, b, c, d \in \mathbb{Z}$ and $x \in K$, where the norm of the tuple is $$\langle (a, b, x, c, d) , (a, b, x, c, d) \rangle / 2 = ad - bc - \langle x, x \rangle / 2,$$ and $\langle -, - \rangle$ also denotes the (positive definite) inner product on $K$. \\
For any elements $\tau, w \in \mathbb{H}$ and any $z \in K_{\mathbb{C}}$, the element $$Z = (1, \tau, z, w,  \tau w + \langle z,z \rangle / 2)$$ lies in $\mathcal{A}$ if and only if it is oriented (i.e., $\mathrm{im}(z)$ lies in the so-called positive cone) and if it satisfies \begin{align*} 0 &< \langle Z, \overline{Z} \rangle \\ &= \tau w + \langle z,z \rangle /2 + \overline{\tau w} + \langle \overline{z}, \overline{z} \rangle / 2 - \tau \overline{w} - \overline{\tau} w - \langle z, \overline{z} \rangle \\ &= -4 \mathrm{im}(\tau) \mathrm{im}(w) - 2 \langle \mathrm{im}(z), \mathrm{im}(z) \rangle, \end{align*} i.e., we have $\langle \mathrm{im}(z), \mathrm{im}(z) \rangle / 2 < \mathrm{im}(\tau) \cdot \mathrm{im}(w)$.

The Fourier-Jacobi expansion of $f$ is defined by $$f(Z) = \sum_{n=0}^{\infty} \phi_n(\tau, z) e^{2\pi i n w}.$$ Each $\phi_n$ is a Jacobi form of index $K(n)$, i.e., $K$ with its inner product multiplied by $n$. \\

The Gritsenko lift is defined in terms of the Hecke raising operators $$V_N \colon J_{k, K} \longrightarrow J_{k, K(N)},$$ which are lattice-index generalizations of the operators introduced in \cite{EZ85}. For $N \ge 1$ we define $$\phi \Big| V_N(\tau, z) = \frac{1}{N} \sum_{\substack{ad = N \\ a > 0}} \sum_{b \in \mathbb{Z}/d\mathbb{Z}} a^k \phi \Big( \frac{a \tau + b}{d}, a z \Big).$$
If $\phi(\tau, z) = \sum_{n, r} c(n, r) q^n \zeta^r$ then $$\phi \Big| V_N(\tau, z) = \sum_{d | (n, r, N)} d^{k-1} c \Big( \frac{Nn}{d^2}, \frac{r}{d} \Big).$$

For $N = 0$ and even $k \ge 4$ we define $\phi \Big| V_0$ to be the abelian function $$\phi \Big| V_0 = -c(0,0) \cdot \frac{B_k}{2k} E_k(\tau) + \frac{1}{(2\pi i)^k} \sum_{r > 0} c(0, r) \wp^{(k-2)}(\tau, \langle r, z \rangle),$$ where $$\wp(\tau, z) = \frac{1}{z^2} + \sum_{\substack{m,n \in \mathbb{Z} \\ (m, n) \ne (0, 0)}} \Big( \frac{1}{(z - m \tau - n)^2} - \frac{1}{(m \tau + n)^2} \Big)$$ is the Weierstrass $\wp$-function, where $\wp^{(k-2)}$ is its $(k-2)$nd derivative with respect to $z$, $B_k$ is the $k$th Bernoulli number, and $E_k$ is the Eisenstein series of weight $k$ for $\mathrm{SL}_2(\mathbb{Z})$, and where $r > 0$ means that $r|_K$ lies in the positive cone. For $k=2$ we define $$\phi \Big| V_0 = \frac{1}{(2\pi i)^k} \sum_{r > 0} c(0, r) \wp(\tau, \langle r, z \rangle).$$

\begin{thm}\label{th:gritsenko}

Suppose $\phi(\tau, z) = \sum_{n, r} c(n, r) q^n \zeta^r$ is a weakly holomorphic Jacobi form of weight $k \ge 1$ and index $K$. Then $$\mathrm{G}(\phi)(Z) := \sum_{m=0}^{\infty} \Big( \phi \Big| V_m \Big) e^{2\pi i m w}$$ defines a meromorphic modular form of weight $k$ on the discriminant kernel $\mathrm{\widetilde O}(L)$. The poles of $\mathrm{G}(\phi)$ are all of order $k$ and lie on hyperplanes $$r^{\perp} = \Big\{ Z \in \mathcal{A}: \; \langle Z, r\rangle = 0 \Big\}$$ for positive norm vectors $r \in L'$. \\
Moreover if $\phi$ is invariant under $\mathrm{Aut}(K'/K)$ then $\mathrm{G}(\phi)$ is modular under the larger group $\mathrm{O}^+(L)$.
\end{thm}
This lift is constructed in Section 14 of \cite{B98}. We call $\mathrm{G}(\phi)$ the \textbf{Gritsenko lift} or additive theta lift of $\phi$ as the case where $\phi$ is holomorphic is due to Gritsenko \cite{G88}. If $r$ is a vector of the form $(0, n, r_K, 1, 0)$ with $r_K = r|_K \in K'$, then it follows from the results of \cite{B98} that in a neighborhood of $r^{\perp}$ the Laurent expansion of $\mathrm{G}(\phi)$ begins \begin{equation}\label{eqn:ppart}\mathrm{G}(\phi)(Z) = \left( \sum_{\lambda=1}^{\infty} c \Big( \lambda^2 n, \lambda r_K \Big) \cdot \frac{(k-1)!}{(2\pi i \lambda)^k} \right) \cdot \frac{1}{\langle r, Z \rangle^k} + h(Z),\end{equation} where $h$ is holomorphic. Under $\mathrm{\widetilde O}(L)$ every hyperplane $r^{\perp}$ is equivalent to one of this form. The only coefficients that appear in Equation (\ref{eqn:ppart}) are singular coefficients of $\phi$, and the sum is finite because $c(n, r)$ is zero whenever $(n, r)$ has sufficiently negative hyperbolic norm. \\

The more famous construction of \cite{B98} is essentially the exponential of Theorem \ref{th:gritsenko} applied to weight $k = 0$.

\begin{defn}
Let $\phi(\tau, z) = \sum_{n, r} c(n, r) q^n \zeta^r$ be a weakly holomorphic Jacobi form of weight $0$ and index $K$ in which all Fourier coefficients $c(n, r)$ are integers. The \textbf{theta block} associated to $\phi$ is the infinite product
\begin{align*} \Theta_{\phi}(\tau, z) &= \eta(\tau)^{c(0, 0)} \cdot \prod_{r > 0} \Big( \frac{\theta_{11}(\tau, \langle r, z \rangle)}{\eta(\tau)} \Big)^{c(0, r)} \\ &= q^{\frac{1}{24} \sum_{r \in K'} c(0, r)} \prod_{r > 0} (\zeta^{r/2} - \zeta^{-r/2}) \times \prod_{r \in K'} \prod_{n=1}^{\infty} (1 - q^n \zeta^r)^{c(0, r)}, \end{align*} where $$\eta(\tau) = q^{1/24} \prod_{n=1}^{\infty} (1 - q^n)$$ is the Dedekind eta function and where $$\theta_{11}(\tau, z) = \sum_{n \in \frac{1}{2}+ \mathbb{Z}} (-1)^{n-1/2} q^{n^2 / 2} e^{2\pi i n z}$$ is the odd Jacobi theta function.
\end{defn}

\begin{thm} \label{th:borcherds}
Suppose $\phi(\tau, z) = \sum_{n, r} c(n, r) q^n \zeta^r$ is a weakly holomorphic Jacobi form of weight $0$ and index $K$ in which all Fourier coefficients $c(n, r)$ are integers.

Then $$\Psi(Z) := \Theta_{\phi}(\tau, z) e^{2\pi i C w} \times \exp \Big( \sum_{m=1}^{\infty} \phi \Big| V_m  \cdot e^{2\pi i m w} \Big)$$ converges on an open subset of $\mathcal{A}$ and it is the Fourier-Jacobi expansion of a meromorphic modular form $\mathrm{B}(\phi)$ of weight $\frac{1}{2}c(0, 0)$ on the discriminant kernel $\mathrm{\widetilde O}(L)$. The zeros and poles of $\mathrm{B}(\phi)$ all lie on hyperplanes $r^{\perp}$. Here, $C$ is the constant $$C = \frac{1}{\mathrm{rank}(K)} \cdot \sum_{r \in K'} c(0, r) \frac{\langle r, r\rangle}{2}.$$
If $r$ is a vector of the form $(0, n, r_K, 1, 0)$ then the order of the zero or pole of $\Psi$ along $r^{\perp}$ is $$\mathrm{ord} \Big( \Psi; r^{\perp} \Big) = \sum_{\lambda =1}^{\infty} c \Big( \lambda^2 n, \lambda r_K \Big).$$
\end{thm}
This form is constructed in Section 13 of \cite{B98} and it is called the \textbf{Borcherds lift} of $\phi$. The setup in terms of Jacobi forms follows the discussion of \cite{WW21}.

\section{Certain algebras of modular forms}

Modular forms with respect to the orthogonal groups of the lattices $H \oplus H \oplus A_1(-1)^{\oplus n}$ will play an important role in later sections, as these occur as transcendental lattices of K3 surfaces polarized by $S$ in (\ref{eqn:lattices}). These were studied by Woitalla in \cite{MR3883325} and they are the special case of root systems of type $B$ in \cite{MR4130466}. The generators of the algebra of modular forms can be chosen to have a particularly simple structure. \\

The root lattice of the simple Lie algebra of type $B_n$ is $A_1^{\oplus n}$, i.e., $\mathbb{Z}^n$ with the bilinear form $$\langle x, y \rangle = 2 \cdot \sum_{i=1}^n x_i y_i.$$ The Weyl group $W(B_n)$ is then the full orthogonal group of $A_1^{\oplus n}$.

Jacobi forms of index $A_1^{\oplus n}$ can be constructed from rank one Jacobi forms by ``exterior" multiplication. Let $$\phi_{-2, 1}(\tau, z) = \frac{\theta_{11}(\tau, z)^2}{\eta(\tau)^6} = (\zeta - 2 + \zeta^{-1}) \prod_{n=1}^{\infty} \frac{(1 - q^n \zeta)^2 (1 - q^n \zeta^{-1})^2}{(1 - q^n)^4}, \quad q = e^{2\pi i \tau}, \zeta = e^{2\pi i z}$$ and $$\phi_{0, 1}(\tau, z) = -\frac{3}{\pi^2}\wp(\tau, z) \cdot \phi_{-2, 1}(\tau, z)$$ be the basic weak Jacobi forms, which generate the $\mathbb{C}[E_4, E_6]$-module of even-weight weak Jacobi forms, as defined in Chapter 9 of \cite{EZ85}. \\

It follows from Theorem 2.4 of \cite{WW23} that the products $$\phi_{*, 1}(\tau, z_1) \cdot ... \cdot \phi_{*, 1}(\tau, z_n), \qquad \text{with} \ * \in \{-2, 0\}$$ generate the $\mathbb{C}[E_4, E_6]$-module of weak Jacobi forms of lattice index $A_1^{\oplus n}$ which are invariant under substitutions $z_i \mapsto -z_i$, and therefore the Weyl group $W(n A_1)$. \\

For $0 \le k \le n$, the symmetric expressions $$f_k(\tau, z_1,...,z_n) := \sum_{\substack{I \subseteq \{1,...,n\} \\ \#I = k}} \Big( \prod_{i \in I} \phi_{-2, 1}(\tau, z_i) \Big) \Big( \prod_{i \notin I} \phi_{0, 1}(\tau, z_i) \Big)$$ are weak Jacobi forms of weight $-2k$ and index $K$ that are invariant under the Weyl group $W(B_n)$. Equivalently, we can write $$f_k(\tau, z_1,...,z_n) = \Big( -\frac{3}{\pi^2} \Big)^k \Phi(\tau, z_1,...,z_n) \cdot \sigma_k \Big( \wp(\tau, z_1),..., \wp(\tau, z_n) \Big),$$ where $$\Phi(\tau, z_1) = \phi_{-2, 1}(\tau, z_1) \cdot ... \cdot \phi_{-2, 1}(\tau, z_n)$$ and where $\sigma_k$ is the $k$th elementary symmetric polynomial. \\

Now let $L_n = H \oplus H \oplus A_1(-1)^{\oplus n}$. The following structure theorem holds (cf. \cite{MR3883325}, \cite{MR4130466}):

\begin{thm}\label{thm:Bn} For $1 \le n \le 4$, the algebra of modular forms $M_*(\mathrm{O}^+(L_n))$ is generated by the Eisenstein series $E_4$ and $E_6$ and by the theta lifts $$\phi_{12 - 2k} := \mathrm{G}\Big( \Delta \cdot f_k \Big), \quad 0 \le k \le n,$$ where $\Delta(\tau) = q \prod_{n=1}^{\infty} (1 - q^n)^{24}$ is the cusp form of weight $12$. Moreover, the generator $\phi_{12 - 2n}$ has a representation as a Borcherds product: $$\phi_{12 - 2n} = \mathrm{B}(\psi_n),$$ where $\psi_n$ is the weak Jacobi form of weight $0$ defined by \begin{align*} \psi_n(\tau, z_1,...,z_n) &:= \frac{1}{2 \cdot \sqrt{\Delta(\tau)}} \Big( \theta_{00}(\tau, z_1) ... \theta_{00}(\tau, z_n) \theta_{00}(\tau)^{12 - n} \\ &\quad\quad\quad\quad\quad\quad - \theta_{01}(\tau, z_1)...\theta_{01}(\tau, z_n) \theta_{01}(\tau)^{12 - n} \\ &\quad\quad\quad\quad\quad\quad - \theta_{10}(\tau, z_1)...\theta_{10}(\tau, z_n) \theta_{10}(\tau)^{12-n} \Big). \end{align*}
\end{thm}

The identity $\phi_{12 - 2n} = \mathrm{B}(\psi_n)$ and the expression in terms of theta-nulls is a special case of a general identity between additive and multiplicative lifts that was proved in \cite{WW21}.
Here $\theta_{00}$, $\theta_{01}$ and $\theta_{10}$ are the even Jacobi theta functions \begin{align*} \theta_{00}(\tau, z) &= \sum_{n = -\infty}^{\infty} q^{n^2 / 2} \zeta^n \\ \theta_{01}(\tau, z) &= \sum_{n=-\infty}^{\infty} (-1)^n q^{n^2 / 2} \zeta^n \\ \theta_{10}(\tau, z) &= \sum_{n \in \frac{1}{2}+\mathbb{Z}} q^{n^2 / 2} \zeta^n, \end{align*} and by abuse of notation we also write $\theta_{ij}(\tau) = \theta_{ij}(\tau, 0)$ for the theta-nulls. \\

From $\phi_{12-2n} = \mathrm{B}(\psi_n)$ one can conclude that $\phi_{12 - 2n}$ has divisor exactly $$\mathrm{div}(\phi_{12 - 2n}) = 2 \cdot \sum_{\substack{r \in L_n' \\ \langle r, r \rangle = 1/2}} r^{\perp},$$ i.e., double zeros on hyperplanes $r^{\perp}$ with $\langle r,r \rangle = 1/2$. \\

When $n = 5$, the algebra of modular forms is more complicated but it is contained in an algebra of meromorphic forms that has a simple description. There are two classes of hyperplanes $r^{\perp}$ with $\langle r, r \rangle = 1/2$ that are inequivalent under the action of $\mathrm{O}^+(L)$. Using the bilinear form $$\langle (a, b, x_1,...,x_5, c, d), (a, b, x_1,...,x_5, c, d) \rangle = ad - bc + 2 (x_1^2 + ... + x_5^2),$$ we have the following representatives: \\
(1) $r_1 = (0, 0, 1/2, 0, 0, 0, 0, 0, 0)$; \\
(2) $r_2 = (0, 1, 1/2, 1/2, 1/2, 1/2, 1/2, 1, 0)$.\\

\begin{thm}\label{thm:B5} The algebra of meromorphic modular forms $M_*^!(\mathrm{O}^+(L_5))$ with poles on $r_2^{\perp}$ and its orbit under $\mathrm{O}^+(L_5)$ is generated by the Eisenstein series $E_4$ and $E_6$ and by the meromorphic theta lifts $$\phi_{12 - 2k} := \mathrm{G} \Big( \Delta \cdot f_k \Big), \quad 0 \le k \le 5.$$ Moreover, the generator $\phi_{2}$ has a representation as a Borcherds product: $$\phi_2 = \mathrm{B}(\psi_5),$$ where $\psi_5$ was defined above.
\end{thm}
In particular, $\phi_2$ has divisor exactly $$\mathrm{div}(\phi_2) = 2 \cdot \sum_{\substack{r \in L_n' \\ \langle r, r \rangle = 1/2 \\ r \; \text{eq. to} \; r_1}} r^{\perp} - 2 \cdot \sum_{\substack{r \in L_n' \\ \langle r, r \rangle = 1/2 \\ r \; \text{eq. to} \; r_2}}.$$

The structure of $M_*^!(L_5)$ was computed in Theorem 5.24 of \cite{WW21}.

\section{Computation of van Geemen--Sarti equations}

We now determine the coefficients of five of the van Geemen--Sarti equations (\ref{eqn:vgs_a_intro}) in terms of the generators of $M_*(\mathrm{O}^+(L_n))$ described in the previous section.
The argument is an induction on $n=8-k$, beginning with $n=0$ (such that $K_n = \{0\}$ and $L = H \oplus H$); the general K3 surface in the associated family has Picard rank 18.

\subsection{Rank 18.}

Since the generic transcendental lattice in this family is $$L = L_0 := H \oplus H,$$ the Jacobi forms that lift to modular forms on the orthogonal group are Jacobi forms of index $\{0\}$, i.e., elliptic modular forms of level one.\\

The lattice $L$ can be identified with the lattice of matrices $$\begin{pmatrix} a & b \\ c & d \end{pmatrix}, \quad a, b, c, d \in \mathbb{Z}$$ with quadratic form given by the determinant. The spinor kernel $\mathrm{O}^+(L)$ contains two copies of $\mathrm{SL}_2(\mathbb{Z})$, one acting by multiplication from the left, the other from the right, and in fact it is generated by these two copies of $\mathrm{SL}_2(\mathbb{Z})$ together with the transpose (cf. \cite{FH00}*{proof of Lemma 4.4}). \\

This implies that modular forms for $F \in M_k(\mathrm{O}^+(L))$ are linear combinations of products $$f(\tau_1) g(\tau_2), \quad f, g \in M_k(\mathrm{SL}_2(\mathbb{Z}))$$ that are symmetric in the two variables $\tau_1, \tau_2$,  i.e., $$M_*(\mathrm{O}^+(L)) = \mathrm{Sym}^2(M_*(\mathrm{SL}_2(\mathbb{Z}))).$$ It is not hard to see that the graded ring $M_*(\mathrm{O}^+(L))$ is therefore a free algebra on three generators, $$M_*(\mathrm{O}^+(L)) = \mathbb{C} \Big[ E_4 \otimes E_4, E_6 \otimes E_6, \Delta \otimes \Delta \Big].$$ %and that $M_*(\widetilde{\mathrm{O}}(L))$ is generated over $M_*(\mathrm{O}^+(L))$ by including the antisymmetric form $$F(\tau_1, \tau_2) = \Delta(\tau_1) \Delta(\tau_2) \Big( j(\tau_1) - j(\tau_2) \Big).$$ Note that $$F = \mathrm{B}(\psi)$$ for the ``Jacobi form'' $$\psi(\tau) = j(\tau) - 720 = q^{-1} + 24 + 196884q + ...$$ of weight zero. \\

The general K3 surface in the family --according to Proposition~\ref{prop:weierstrass}-- is described by the equation \begin{equation}\label{eqn:K3rank18} y^2 = x^3 + A(t) x^2 + B(t) x, \end{equation} where $A(t) = t^3 + a_4 t + a_6$ and where in this case $B = b_{12}$ is constant.
The discriminant of the right-hand side of Equation (\ref{eqn:K3rank18}) is $B^2 (A^2 - 4B)$, which itself is a polynomial of $t$ with discriminant $$D = 4096 b_{12}^3 \cdot \Big( 16a_4^6 + 216 a_4^3 a_6^2 + 729 a_6^4 + 864 a_4^3 b_{12} - 5832a_6^2 b_{12} + 11664 b_{12}^2 \Big).$$
Here,   $a_4$, $a_6$ and $b_{12}$ are holomorphic modular forms whose weight is indicated by the subscript, and due to the ring structure they are constant multiples of $E_4 \otimes E_4$, $E_6 \otimes E_6$ and $\Delta \otimes \Delta$, respectively (as $a_4$ and $a_6$ live in a one-dimensional space of modular forms, and $B = b_{12}$ is never allowed to vanish). \\

The weight 24 form $D / b_{12}^3$ is a modular form that vanishes on the discriminant locus $$\Big\{ (\tau_1, \tau_2): \; \mathrm{SL}_2(\mathbb{Z}) \cdot \tau_1 = \mathrm{SL}_2(\mathbb{Z}) \cdot \tau_2 \Big\}$$ and is therefore a constant multiple of the form $$F(\tau_1, \tau_2) := \Delta(\tau_1)^2 \Delta(\tau_2)^2 \Big( j(\tau_1) - j(\tau_2) \Big)^2,$$ where $j$ is the $j$-invariant. (Note that $F = \mathrm{B}(\psi)$ for the Jacobi form $$\psi(\tau) = 2j(\tau) - 1440 = 2q^{-1} + 48 + 2 \cdot 196884q + O(q^2)$$ of weight zero.) By comparing coefficients between that form and $$16a_4^2 + 216a_4^3 a_6^2 + 729a_6^4 + 864a_4^3 b_{12} - 5832 a_6^2 b_{12} + 11664b_{12}^2,$$ we find that if $b_{12} = C \cdot \Delta \otimes \Delta$ then we must have $$a_4 = -\frac{1}{48} C^{1/3} \cdot E_4 \otimes E_4 \quad \text{and} \quad a_6 = \frac{1}{864} C^{1/2} \cdot E_6 \otimes E_6.$$
The parameter $C$ is arbitrary, but it is natural to fix $C = 12^6$ such that the van Geemen-Sarti equation simplifies to 

\begin{equation} \label{eqn:rank18sol} y^2 = x^3 + \Big( t^3 - 3 (E_4 \otimes E_4) t + 2 (E_6 \otimes E_6) \Big)x^2 + 12^6 (\Delta \otimes \Delta ) x. \end{equation}

\subsection{Rank 17.}

The transcendental lattice of the general member of the family is now $$L = L_1 := H \oplus H \oplus A_1(-1).$$ There is a well-known exceptional isomorphism $$\mathrm{Spin}(L) = \mathrm{Sp}_4(\mathbb{Z})$$ that makes it possible to identify modular forms on $\mathrm{O}^+(L)$ with Siegel modular forms of degree two and even weight (and modular forms on $\widetilde{\mathrm{O}}(L)$ with Siegel modular forms of any weight). By a famous theorem of Igusa \cite{Igusa62}, the ring structure is $$M_*(\mathrm{O}^+(L)) = \mathbb{C}[E_4, E_6, \chi_{10}, \chi_{12}],$$ where $E_4, E_6$ are the Siegel Eisenstein series and where the cusp forms $\chi_{10}, \chi_{12}$ are easiest to describe as Gritsenko lifts: $$\chi_{10} = \mathrm{G}(\Delta \cdot \phi_{-2, 1}), \quad \chi_{12} = \mathrm{G}(\Delta \cdot \phi_{0, 1}).$$
(Note that these Gritsenko lifts differ from Igusa's definitions \cite{Igusa62} by multiples of 4 and 12, respectively.)

There is a restriction map, $$\mathrm{res} \colon \   M_*(\mathrm{O}^+(L_1)) \longrightarrow M_*(\mathrm{O}^+(L_0))$$ which through the interpretations $M_*(\mathrm{O}^+(L_1)) = M_{2*}(\mathrm{Sp}_4(\mathbb{Z}))$ and $M_*(\mathrm{O}^+(L_0)) = \mathrm{Sym}^2(M_*(\mathrm{SL}_2(\mathbb{Z}))$ is the map $$\mathrm{res}(F)(\tau_1, \tau_2) = F \left( \begin{pmatrix} \tau_1 & 0 \\ 0 & \tau_2 \end{pmatrix} \right), \quad \tau_1, \tau_2 \in \mathbb{H}.$$ The action of the restriction map on the generators is \begin{equation}\label{eqn:res18} \mathrm{res}(E_4) = E_4 \otimes E_4, \quad \mathrm{res}(E_6) = E_6 \otimes E_6, \quad \mathrm{res}(\chi_{10}) = 0, \quad \mathrm{res}(\chi_{12}) = 12 \cdot (\Delta \otimes \Delta). \end{equation}

The general K3 surface in the family --according to Proposition~\ref{prop:weierstrass}-- is described by the equation $$y^2 = x^3 + A(t) x^2 + B(t) x\,,$$ where $A(t) = t^3 + a_4 t + a_6$ and $B(t) = b_{10} t + b_{12}$ and the previous family is cut out by the equation $b_{10} = 0$.
In view of Equations (\ref{eqn:rank18sol}) and (\ref{eqn:res18}), we have $$a_4 = -3E_4, \quad a_6 = 2 E_6, \quad b_{12} = 12^5 \cdot \chi_{12},$$ and there is a constant $C$ such that $b_{10} = C \cdot \chi_{10}$. We compare the first Fourier coefficients between the discriminant of $A^2 - 4B$ and the reflective product $$\Psi_{60}\left( \begin{pmatrix} \tau & z \\ z & w \end{pmatrix} \right) = q^3 s^3 (q - s)^2 (r^{-1} + 2 + r) + O(q, s)^9, \quad q = e^{2\pi i \tau}, r = e^{2\pi i z}, s = e^{2\pi i w},$$ which is the Borcherds lift of the following weakly holomorphic Jacobi form of weight 0 and index 1, $$\psi = \frac{29}{24} \cdot \frac{E_4^2 E_{4, 1}}{\Delta} + \frac{19}{24} \cdot \frac{E_6 E_{6, 1}}{\Delta} = 2q^{-1} + (2 \zeta^{-2} - 2 \zeta^{-1} + 120 - 2 \zeta + 2 \zeta^2) + O(q).$$ We conclude that $C = -12^5$. Hence, we have derived the van Geemen-Sarti equation

\begin{equation} \label{eqn:rank17sol} y^2 = x^3 + \Big( t^3 - 3 E_4 t + 2 E_6 \Big) x^2 + 12^5 \cdot \Big( -\chi_{10} t + \chi_{12} \Big) x.\end{equation}

\subsection{Rank 16.}

The transcendental lattice of the general K3 surface in this family is $$L = L_2 := H \oplus H \oplus A_1(-1)^{\oplus 2},$$ so the Jacobi forms that are lifted to produce modular forms on the orthogonal group have lattice index $A_1^{\oplus 2}$.

By Theorem \ref{thm:Bn}, the algebra of modular forms of level $\mathrm{O}^+(L)$ is $$M_*(B_2) = \mathbb{C}[E_4, E_6, \chi_8, \chi_{10}, \chi_{12}],$$ where $E_4$ and $E_6$ are the Eisenstein series and where $$\chi_8 = \mathrm{G} \Big( \Delta \phi_{-2,1} \otimes \phi_{-2, 1} \Big), \quad \chi_{10} = \mathrm{G} \Big( \Delta \cdot \phi_{-2, 1} \otimes \phi_{0, 1} + \Delta \cdot \phi_{0, 1} \otimes \phi_{-2, 1} \Big),$$ and $$\chi_{12} = \mathrm{G} \Big( \Delta \cdot \phi_{0,1} \otimes \phi_{0, 1} \Big).$$

There are two obvious embeddings of $A_1$ into $A_1^{\oplus 2}$ and they induce the same restriction map $$\mathrm{res}\colon \ M_*(\mathrm{O}^+(L_2)) \longrightarrow M_*(\mathrm{O}^+(L_1)).$$
The images of the generators are $$\mathrm{res}(E_4^{L_2}) = E_4^{L_1}, \; \mathrm{res}(E_6^{L_2}) = E_6^{L_1}, \; \mathrm{res}(\chi_8^{L_2}) = 0, \; \mathrm{res}(\chi_{10}^{L_2}) = 12 \cdot \chi_{10}^{L_1}, \; \mathrm{res}(\chi_{12}^{L_2}) = 12 \cdot \chi_{12}^{L_1},$$ where the superscript indicates to which algebra the form belongs.\\

We will need the following weak Jacobi forms of weight two which are \emph{not} Weyl-invariant and whose Gritsenko lifts are modular only under the discriminant kernel:

\begin{align*}
f_1(\tau, z_1, z_2) &:= \phi_{-2, 1}(\tau, z_1) E_{4, 1}(\tau, z_2) = (\zeta_1^{-1} - 2 + \zeta_1) + O(q); \\
f_2(\tau, z_1, z_2) &:= E_{4, 1}(\tau, z_1) \phi_{-2, 1}(\tau, z_2) = (\zeta_2^{-1} - 2 + \zeta_2) + O(q); \\
\end{align*}
where $\zeta_j = e^{2\pi i z_j}$ and $q = e^{2\pi i \tau}$. 

By Theorem \ref{th:gritsenko}, the Fourier-Jacobi series of the meromorphic Gritsenko lifts $\mathrm{G}(f_i)$ begin $$\mathrm{G}(f_i) = -\frac{1}{4\pi^2} \wp(\tau, z_i) + f_i(\tau, z_1, z_2) s + O(s^2).$$

The general K3 surface in the family --according to Proposition~\ref{prop:weierstrass}-- is described by the equation $$y^2 = x^3 + (t^3 + a_4 t + a_6) x^2 + (b_8 t^2 + b_{10} t + b_{12}) x,$$ where the previous family is cut out by $b_8 = 0$. In view of the ring structure of $M_*(\mathrm{O}^+(L)),$ the forms $a_4, a_6, b_{10}$ are uniquely determined by their values (\ref{eqn:rank17sol}) when $b_8 = 0$. We obtain $$a_4 = -3 E_4, \quad a_6 = 2 E_6, \quad b_{10} = -12^4 \cdot \chi_{10},$$ and moreover $b_8$ is a multiple of $\chi_8$; say $b_8 = A \cdot \chi_8$. Also, when we factor $$B(t) = b_8 (t - \beta_1) (t - \beta_2) = b_4 t^2 - b_8 (\beta_1 + \beta_2) t + b_4 \beta_1 \beta_2,$$ then $\beta_1$ and $\beta_2$ define meromorphic modular forms of weight two on the discriminant kernel with double poles on rational quadratic divisors on which $b_4$ vanishes. Since there are no nonzero holomorphic modular forms of weight two, we obtain $$\beta_1 = C \cdot \mathrm{G}(f_1) \quad \text{and} \quad \beta_2 = C \cdot \mathrm{G}(f_2)$$ for some other constant $C$.

By comparing the leading Fourier-Jacobi coefficients in
\begin{align*} -b_8 (\beta_1 + \beta_2) &= -AC \cdot \chi_8 \cdot (\mathrm{G}(f_1) + \mathrm{G}(f_2)) \\ &= -AC \Big(\Delta \phi_{-2, 1}(\tau, z_1) \phi_{-2, 1}(\tau, z_2) s + O(s^2) \Big) \cdot \Big( -\frac{1}{4\pi^2} \wp(\tau, z_1) - \frac{1}{4\pi^2} \wp(\tau, z_2) + O(s) \Big) \\ &= -\frac{AC}{12} \Delta(\tau) \Big(\phi_{-2,1}(\tau, z_1) \phi_{0,1}(\tau, z_2) + \phi_{0,1}(\tau, z_1) \phi_{-2, 1}(\tau, z_2) \Big) s + O(s^2) \end{align*} and $$b_{10} = -12^4 \cdot \chi_{10} = -12 \Delta(\tau) \cdot \Big( \phi_{-2, 1}(\tau, z_1) \phi_{0, 1}(\tau, z_2) + \phi_{0, 1}(\tau, z_1) \phi_{-2, 1}(\tau, z_2) \Big) s + O(s^2),$$ we conclude that $AC = 12^5$, and by comparing
\begin{align*} b_8 \beta_1 \beta_2 &= A C^2 \chi_8 \mathrm{G}(f_1) \mathrm{G}(f_2) \\ &= A C^2 \Big(\Delta \phi_{-2,1}(\tau, z_1) \phi_{-2, 1}(\tau, z_2) s + O(s^2) \Big) \Big( -\frac{1}{4\pi^2} \wp(\tau, z_1) + f_1(\tau, z_1, z_2) s + O(s^2) \Big) \\ &\quad \quad \times \Big(-\frac{1}{4\pi^2} \wp(\tau, z_2) + f_2(\tau, z_1, z_2) s + O(s^2) \Big) \\ &= \frac{A C^2}{144} \Delta(\tau) \phi_{0, 1}(\tau, z_1) \phi_{0, 1}(\tau, z_2) s \\ &= \frac{AC^2}{144} \phi_{12} \end{align*} with
 $b_{12}$, we conclude that $$b_{12} = \frac{A C^2}{144}\chi_{12} = 12^3 C \cdot \chi_{12}.$$
When $b_8 = 0$, the result of (\ref{eqn:rank17sol}) implies $$12^5 \chi_{12} = \mathrm{res}(b_{12}) = 12^3 C \cdot 12 \chi_{12},$$ i.e., $C = 12$ and therefore $A = 12^4$. Hence, we have derived the van Geemen-Sarti equation

\begin{equation}\label{eqn:rank16sol} y^2 = x^3 + (t^3 - 3 E_4 t + 2 E_6) x^2 + 12^4 \cdot (\chi_8 t^2 - \chi_{10} t + \chi_{12}) x.
\end{equation}

\subsection{Rank 15}

The transcendental lattice of the general surface in the family is $$L = L_3 := H \oplus H \oplus A_1(-1)^{\oplus 3}.$$ For the full orthogonal group, we have the graded ring structure $$M_*(\mathrm{O}^+(L)) = \mathbb{C}[E_4, E_6, \chi_6, \chi_8, \chi_{10}, \chi_{12}]$$ as a special case of Theorem \ref{thm:Bn}, where $E_4$ and $E_6$ are the Eisenstein series and where $\chi_6, \chi_8, \chi_{10}, \chi_{12}$ are defined as Gritsenko lifts: $$\chi_{6 + 2k} = \mathrm{G}(\varphi_{6 + 2k}), \quad k = 0, 1, 2, 3,$$
where $\varphi_{6 + 2k}$ is the Jacobi cusp form $$\varphi_{6+2k}(\tau, z_1, z_2, z_3) = \Delta(\tau) \phi_{-2, 1}(\tau, z_1) \phi_{-2, 1}(\tau, z_2) \phi_{-2, 1}(\tau, z_3) \sigma_k \Big( \frac{3}{\pi^2} \wp(\tau, z_1),..., \frac{3}{\pi^2} \wp(\tau, z_3) \Big),$$ and $\sigma_0,...,\sigma_3$ are the elementary symmetric polynomials $$\sigma_0 = 1, \quad \sigma_1(x,y,z) = x+y+z, \quad \sigma_2(x,y,z) = xy+yz+zx, \quad \sigma_3(x,y,z) = xyz.$$

Up to the action of the discriminant kernel, there are three embeddings of $H \oplus H \oplus A_1(-1)^{\oplus 2}$ into $L$ as hyperplanes $r_1^{\perp}$, $r_2^{\perp}$, $r_3^{\perp}$, where \begin{equation} \label{eqn:rank16r_i} r_1 = (1, 0, 0), \quad r_2 = (0, 1, 0), \quad r_3 = (0, 0, 1) \in A_1(-1)^{\oplus 3}.\end{equation}
The result of restricting the generating forms along any of these embeddings is $$E_4^{L_3} \mapsto E_4^{L_2}, \; E_6^{L_3} \mapsto E_6^{L_2}, \; \chi_6^{L_3} \mapsto 0, \; \chi_8^{L_3} \mapsto 12 \chi_8^{L_2}, \; \chi_{10}^{L_3} \mapsto 12 \chi_{10}^{L_2}, \; \chi_{12}^{L_3} \mapsto 12 \chi_{12}^{L_2}.$$

There are uniquely determined weak Jacobi forms $f_1, f_2, f_3$ of weight two and lattice index $A_1^{\oplus 3}$ whose Fourier expansions begin as
$$f_j(\tau, z_1, z_2, z_3) = (\zeta_j^{-1} - 2 + \zeta_j) + O(q),$$ with $q = e^{2\pi i \tau}$ and $\zeta_j = e^{2\pi i z_j}$.
In particular, each $\mathrm{G}(f_j)$ has double poles precisely on the orbit of $r_j^{\perp}$ under the discriminant kernel, with Fourier-Jacobi expansion beginning $$\mathrm{G}(f_j)(\tau, z_1, z_2, z_3, w) = -\frac{1}{4\pi^2} \wp(\tau, z_j) + f_j(\tau, z_1, z_2, z_3) s + O(s^2), \quad s = e^{2\pi i w}.$$
%There are also (unique) weak Jacobi forms $\psi_6$ and $b_{12}, b_{13}, b_{23}$ whose Fourier expansions begin
%$$\psi_6(\tau, z_1, z_2, z_3) = 2 \zeta_1^{-1} + 2 \zeta_2^{-1} + 2 \zeta_3^{-1} + 12 + 2 \zeta_1 + 2 \zeta_2 + 2 \zeta_3 + O(q)$$ and
%\begin{align*} b_{12}(\tau, z_1, z_2, z_3) &= (\zeta_1^{-1} - 2 + \zeta_1) (\zeta_2^{-1} - 2 + \zeta_2) + O(q); \\
%b_{13}(\tau, z_1, z_2, z_3) &= (\zeta_1^{-1} - 2 + \zeta_1)(\zeta_3^{-1} - 2 + \zeta_3) + O(q); \\
%b_{23}(\tau, z_1, z_2, z_3) &= (\zeta_2^{-1} - 2 + \zeta_2)(\zeta_3^{-1} - 2 + \zeta_3) + O(q). \end{align*}
%The Borcherds product $\mathrm{B}(\psi_6)$ has weight $6$ and divisor $2 r_1^{\perp} + 2r_2^{\perp} + 2 r_3^{\perp}$, and the products $\mathrm{B}(b_{ij})$ are meromorphic of weight $2$ with divisor $(r_i+r_j)^{\perp} -2r_i^{\perp} - 2r_j^{\perp}$.

%Since $\mathrm{G}(f_j)$ has a double pole exactly on $r_j^{\perp}$, and $\mathrm{G}(f_i) - \mathrm{G}(f_j)$ vanishes on $(r_i + r_j)^{\perp}$, we obtain (similarly to the last section) the identities $$\mathrm{G}(f_i) - \mathrm{G}(f_j) = \mathrm{B}(b_{ij}).$$

The general K3 surface in the family --according to Proposition~\ref{prop:weierstrass}-- is described by the equation $$y^2 = x^3 + A(t) x^2 + B(t)x, $$ where $A(t) = t^3 + a_4 t + a_6$ and $B(t) = b_6 t^3 + b_8 t^2 + b_{10} t + b_{12}$, and the rank 16 family of the previous subsection is cut out by the equation $b_6 = 0$. The forms $a_k$, $b_k$ are modular forms of weight $k$ whose restrictions to any of the hyperplanes $r_j^{\perp}$ are given by Equation (\ref{eqn:rank16sol}), and in view of the algebra structure this uniquely determines $$a_2 = -3 E_4, \quad b_4 = 12^3 \chi_8;$$ in addition, $b_3 = A \cdot \chi_6$ for some nonzero constant $A$.\\

If we factor $$B(t) = b_6 (t - \beta_1) (t - \beta_2) (t - \beta_3),$$ then $\beta_1, \beta_2, \beta_3$ are meromorphic modular forms of weight two on the discriminant kernel of $L$ with double poles on $r_1^{\perp}, r_2^{\perp},$ and $r_3^{\perp}$, respectively, and they are permuted under $\mathrm{O}^+(L)$. This forces (possibly after reordering) $$\beta_1 = C \cdot \mathrm{G}(f_1), \; \beta_2 = C \cdot \mathrm{G}(f_2), \; \beta_3 = C \cdot \mathrm{G}(f_3)$$ with a common nonzero constant $C$. \\

Now, we have \begin{align*} 12^3 \cdot \chi_8 &= b_8 \\ &= -b_6 (\beta_1 + \beta_2 + \beta_3) \\ &= \frac{1}{4\pi^2} AC \chi_6 \sum_{i=1}^3 \wp(\tau, z_i) + O(s) \\ &= -\frac{AC}{12} \Delta \phi_{-2, 1}(\tau, z_1)\phi_{-2,1}(\tau, z_2) \phi_{-2, 1}(\tau, z_3) \sum_{i=1}^3 \frac{\phi_{0,1}(\tau, z_i)}{\phi_{-2, 1}(\tau, z_i)} + O(s), \end{align*} and comparing the constant Fourier-Jacobi coefficients on both sides of this equation yields $AC = -12^4$. The Fourier-Jacobi expansions of $b_{10}$ and $b_{12}$ therefore begin as

\begin{align*} b_{10} &= b_6 (\beta_1 \beta_2 + \beta_1 \beta_3 + \beta_2 \beta_3) \\ &= \frac{AC^2}{144} \cdot \Delta(\tau) \phi_{-2, 1}(\tau, z_1) \phi_{-2, 1}(\tau, z_2) \phi_{-2, 1}(\tau, z_3) \sigma_2 \Big(\frac{3}{\pi^2} \wp(\tau, z_1), \frac{3}{\pi^2} \wp(\tau, z_2), \frac{3}{\pi^2} \wp(\tau, z_3) \Big)s + O(s^2) \\ &= 144 C \cdot \varphi_{10}(\tau, z_1, z_2, z_3) s + O(s^2),\end{align*}
which uniquely determines $C = 12$ and $A = -12^3$ and $b_{10} = -1728 \chi_{10},$ and

\begin{align*} b_{12} &= -b_6\beta_1 \beta_2 \beta_3 \\ &= -\frac{AC^3}{1728} \cdot \Delta(\tau) \phi_{-2, 1}(\tau, z_1) \phi_{-2, 1}(\tau, z_2) \phi_{-2, 1}(\tau, z_3) \Big( -\frac{3}{\pi^2} \wp(\tau, z_1) \cdot -\frac{3}{\pi^2} \wp(\tau, z_2) \cdot -\frac{3}{\pi^2}\wp(\tau, z_3) \Big)s \\ &- \frac{AC^3}{144} \Delta(\tau) \phi_{-2, 1}(\tau, z_1) \phi_{-2, 1}(\tau, z_2) \phi_{-2, 1}(\tau, z_3) \sum_{i=1}^3 \Big[f_i(\tau, z_1, z_2, z_3) \prod_{j \ne i} -\frac{3}{\pi^2} \wp(\tau, z_j)\Big] s^2 + O(s^3), \\ &= 12 C^2 \cdot \varphi_{12}(\tau, z_1, z_2, z_3)s - 12C^2 \cdot \Big( \varphi_{12} \Big| V_2 - 2592 \varphi_6^2 \Big) s^2 + O(s^3).\end{align*} The coefficients of $B(t)$ are therefore $$b_6 = -1728 \chi_6, \quad b_8 = 1728 \chi_8,$$ $$b_{10} = -1728 \chi_{10}, \quad b_{12} = 1728 \Big( \chi_{12} - 2592 \cdot \chi_6^2 \Big).$$

Only the coefficient $a_6$ remains to be determined. Since the value at $b_6 = 0$ is determined by Equation (\ref{eqn:rank16sol}), we have $$a_6 = 2 E_6 + C \cdot \chi_6$$ for some constant $C$.
The contant $C$ can theoretically be determined by computing the discriminant form (as in the derivation of Equation (\ref{eqn:rank17sol}) but this is computationally difficult. Instead, we use a method that we call \emph{paramodular restriction}, which will also be useful in the later sections. It is based on the following observation:

\begin{lem}\label{lem:rank15}
Write $$B(t) = b_6 t^3 + b_8 t^2 + b_{10} t + b_{12} = b_6 (t - \beta_1) (t - \beta_2)(t - \beta_3),$$ where each $\beta_i$ has poles exactly on the $\mathrm{\widetilde O}(L)$-orbit of $r_i^{\perp}$. Then $\beta_i - \beta_j$ vanishes precisely on the $\mathrm{\widetilde O}(L)$-orbit of $(r_i + r_j)^{\perp}$.
\end{lem}
\begin{proof}
Since the reflection through the hyperplane $(r_i + r_j)^{\perp}$ swaps $r_i^{\perp}$ and $r_j^{\perp}$, it maps $\beta_i - \beta_j$ to $\beta_j - \beta_i$. This implies that $\beta_i - \beta_j$ vanishes on $(r_i + r_j)^{\perp}$. \\
One can show that there are unique weak Jacobi forms of weight 0 and lattice index $A_1^{\oplus 3}$ whose Fourier series begin
\begin{align*} b_{12}(\tau, z_1, z_2, z_3) &= (\zeta_1^{-1} - 2 + \zeta_1) (\zeta_2^{-1} - 2 + \zeta_2) + O(q); \\
b_{13}(\tau, z_1, z_2, z_3) &= (\zeta_1^{-1} - 2 + \zeta_1)(\zeta_3^{-1} - 2 + \zeta_3) + O(q); \\
b_{23}(\tau, z_1, z_2, z_3) &= (\zeta_2^{-1} - 2 + \zeta_2)(\zeta_3^{-1} - 2 + \zeta_3) + O(q). \end{align*}
The Borcherds products $\psi_{ij} = \mathrm{B}(b_{ij})$ have double poles on $r_i^{\perp}$ and $r_j^{\perp}$ and vanish on $(r_i + r_j)^{\perp}$ by construction. This means that $(\beta_i - \beta_j) / \psi_{ij}$ are holomorphic modular forms of weight $0$ and hence (nonzero) constants, so \[ \mathrm{div}(\beta_i - \beta_j) = \mathrm{div}(\psi_{ij}) = (r_i + r_j)^{\perp} - 2 r_i^{\perp} - 2r_j^{\perp}. \qedhere\]
\end{proof}

We consider the specialization of the family $$y^2 = x^3 + (t^3 + a_4 t + a_6) x^2 + \Big(b_6 \prod_{i=1}^3 (t - \beta_i) \Big) x$$
to $\beta_1 = \beta_2 = \beta_3 =: \beta$. In view of Lemma \ref{lem:rank15}, the general member of the specialized family has transcendental lattice $H \oplus H \oplus A_1(-3).$
The orthogonal group of that lattice is essentially the paramodular group of level 3; see e.g. Section 1.3 of \cite{GN98}.

The discriminant of $$y^2 = x^3 + (t^3 + a_4 t + a_6) x^2 + b_6 (t - \beta)^3 x$$ is given by $$b_6^2 (t - \beta)^6 \Big( t^6 + 2 a_4 t^4 + 2(a_6 - 2 b_6) t^3 + (a_4^2 + 12 b_6 \beta) t^2 + (2 a_4 a_6 - 12 b_6 \beta^2) t + (4 b_6 \beta^3 + a_6^2) \Big),$$ and the irreducible factor of degree six of that has discriminant $$D := b_6^3 \cdot (\beta^3 + a_4 \beta + a_6)^3 \cdot P$$ with a modular form $P$ of weight $36$ whose expression is too long to write out here. \\

The presence of $b_6 \cdot (\beta^3 + a_4 \beta + a_6)$ in the discriminant $D$ implies that it is a reflective Borcherds product of weight $12$ with a double zero on the paramodular Humbert surface of discriminant $4$ in the sense of \cite{GN98}, and this in fact determines it uniquely: we obtain $$b_6 \cdot (\beta^3 + a_4 \beta + a_6) = b_6 \beta^3 + a_4 \beta + \Big( \frac{6}{17} b_6^2 + 2 E_6 b_6 \Big)$$ and therefore $$a_6 = 2 E_6 + \frac{6}{17} b_6 = 2 E_6 + \frac{10368}{17} \chi_6.$$

Altogether, we have derived the van Geemen-Sarti equation
\begin{equation} \label{eqn:rank15sol}
y^2 = x^3 + \Big( t^3 - 3 E_4 t + 2 E_6 + \frac{10368}{17} \chi_6 \Big) x^2 + 1728 \Big( -\chi_6 t^3 + \chi_8 t^2 - \chi_{10} t + \chi_{12} - 2592 \chi_6^2 \Big)x.
\end{equation}

\subsection{Rank 14}

The transcendental lattice of the generic K3 surface in the family is $$L = L_4 := H \oplus H \oplus A_1(-1)^{\oplus 4}.$$ By Theorem \ref{thm:Bn}, the algebra structure of modular forms for the full orthogonal group is $$M_*(\mathrm{O}^+(L)) = \mathbb{C}[E_4, E_6, \chi_4, \chi_6, \chi_8, \chi_{10}, \chi_{12}],$$ where $E_4, E_6$ are Eisenstein series and where $\chi_{4+2k} = \mathrm{G}(\varphi_{4+2k})$ for the Jacobi form $$\varphi_{4+2k}(\tau, z_1,...,z_4) = \Delta(\tau) \cdot \Big( \prod_{i=1}^4 \phi_{-2, 1}(\tau, z_i) \Big) \cdot \sigma_k \Big( \frac{3}{\pi^2} \wp(\tau, z_1),..., \frac{3}{\pi^2} \wp(\tau, z_4) \Big), \quad k=0,1,2,3,4.$$ Here, $\sigma_k$ is the usual elementary symmetric polynomial. In contrast to the earlier sections, $\varphi_{4+2k}$ and therefore $\chi_{4+2k}$ are not cusp forms. \\

Similarly to the previous section, there are four natural embeddings of $H \oplus H \oplus A_1(-1)^{\oplus 3}$ into $L$ as the hyperplanes $r_i^{\perp},$ $i=1,...,4$ where $$r_i = (0,..., 1,...,0) \in A_1(-1)^{\oplus 4}$$ has $1$ in the $i^{\text{th}}$ entry and $0$ elsewhere, and along any of the embeddings the generating modular forms have the following restrictions: 

$$E_4^{L_4} \mapsto E_4^{L_4}, \quad E_6^{L_3} \mapsto E_6^{L_3} + \frac{1512}{17} \chi_6^{L_3};$$

$$\chi_4^{L_4} \mapsto 0, \; \chi_6^{L_4} \mapsto 12 \chi_6^{L_3}, \; \chi_8^{L_4} \mapsto 12 \chi_8^{L_3}, \; \chi_{10}^{L_4} \mapsto 12 \chi_{10}^{L_3}, \; \chi_{12}^{L_4} \mapsto 12 \chi_{12}^{L_3}.$$

The general K3 surface in the family --according to Proposition~\ref{prop:weierstrass}-- is described by the equation $$y^2 = x^3 + A(t) x^2 + B(t) x,$$ where $A(t) = t^3 + a_4 t + a_6$ and $B(t) = b_4 t^4 + b_6 t^3 + b_8 t^2 + b_{10} t + b_{12}$.

Using the ring structure and the values from Equation (\ref{eqn:rank15sol}), we see immediately that $$a_6 = 2 E_6 + 36 \chi_6 \quad \text{and} \quad b_6 = -144 \chi_6.$$

Arguing as in the previous sections, we factor $$B(t) = b_4 (t - \beta_1) (t - \beta_2) (t - \beta_3) (t - \beta_4),$$
where $\beta_i$ are meromorphic modular forms on the discriminant kernel of $L$ with double poles on the hyperplane $r_i^{\perp}$. There are unique weak Jacobi forms of weight two with Fourier series beginning $$f_i = (\zeta_i^{-1} - 2 + \zeta_i) + O(q), \quad i=1,...,4,$$ and whose Gritsenko lifts have Fourier-Jacobi expansions beginning $$\mathrm{G}(f_i) = -\frac{1}{4\pi^2} \wp(\tau, z_i) + f_i(\tau, z_1,...,z_4) s + O(s^2);$$ we must have $\beta_i = C \cdot \mathrm{G}(f_i)$ for some constant $C$.
Also, there is a constant $A$ such that $b_4 = A \cdot \chi_4$.

Now, we consider the Fourier-Jacobi expansion of $b_6$:

\begin{align*} -144 \chi_6 &= -b_4 (\beta_1 + \beta_2 + \beta_3 + \beta_4) \\ &= \frac{1}{4\pi^2} AC \chi_4 \cdot \Big( \sum_{i=1}^4 \wp(\tau, z_i) + O(s) \Big) \\ &= -\frac{AC}{12} \Delta \Big( \prod_{i=1}^4 \phi_{-2, 1}(\tau, z_i) \Big) \sum_{i=1}^4 \frac{\phi_{0, 1}(\tau, z_i)}{\phi_{-2, 1}(\tau, z_i)} + \mathrm{O}(s). \end{align*} By comparing constant terms we obtain $AC = 12^3$. From the Fourier-Jacobi expansions of the other terms, we obtain $$b_8 = b_4 \prod_{i<j} \beta_i \beta_j = \frac{AC^2}{12^2} (\chi_8 - 432 \chi_4^2) = 12C \cdot (\chi_8 - 432 \chi_4^2)$$ and $$b_{10} = -b_4 \prod_{i < j < k} \beta_i \beta_j \beta_k = -\frac{AC^3}{12^3} (\chi_{10} - 432 \chi_4 \chi_6)$$ and finally $$b_{12} = b_2 \beta_1 \beta_2 \beta_3 \beta_4 = \frac{AC^4}{12^4} \Big( \chi_{12} + 360 \chi_4 \chi_8 - 216 \chi_6^2 - 235008 \chi_4^3 - 648 E_4 \chi_4^2 \Big).$$
By comparing this with the expressions in Equation (\ref{eqn:rank15sol}), we obtain $A = 12^2$ and $C = 12$ and therefore $$b_8 = 144 (\chi_8 - 432 \chi_4^2),$$ $$b_{10} = -144(\chi_{10} - 432 \chi_4 \chi_6)$$ and $$b_{12} = 144 (\chi_{12} + 360 \chi_4 \chi_8 - 216 \chi_6^2 - 235008 \chi_4^3 - 648 E_4 \chi_4^2).$$

The only coefficient left undetermined is $a_4$: we have $$a_4 = -3 E_4 + C \cdot \chi_4$$ for some constant $C$. The remaining constant can be determined by paramodular restriction as in the previous subsection. Setting $\beta_1 = \beta_2 = \beta_3 = \beta_4 =: \beta$ specializes the transcendental lattice to $H \oplus H \oplus A_1(-4)$, and the specialized coefficients become paramodular forms on the group $K(4)$. The irreducible degree six factor of the discriminant of the van Geemen--Sarti equation $$y^2 = x^3 + \big(t^3 + a_4 t + a_6\big) x^2 + \big(t - \beta\big)^4 x$$ itself has discriminant $$D = b_4^3 (a_6 + a_4 \beta + \beta^3)^4 \cdot P$$ with a factor $P$ of weight $24$. The expression $a_6 + a_4 \beta + \beta^3$ is therefore a reflective paramodular form of weight $6$, and in fact $$a_6 + a_4 \beta + \beta^3 = \frac{\mathrm{B}(\psi_{9/2})^2}{\mathrm{B}(\psi_{1/2})^6},$$ where $\mathrm{B}(\psi_{1/2})$ and $\mathrm{B}(\psi_{9/2})$ are (the unique) reflective Borcherds products with simple zeros on the Heegner divisors of discriminant $1/16$ and $1/4$, respectively. Using the formulas for $a_6$ and $\beta$ that were determined earlier, we find $$a_4 = -3E_4 - 144 \chi_4.$$ Altogether, we have derived the van Geemen-Sarti equation

\begin{align*} y^2 &= x^3 + \Big( t^3 - 3 E_4 t - 144 \chi_4 t + 2 E_6 + 432 \chi_6 \Big) x^2 \\ &+ 144 \Big( \chi_4 t^4 - \chi_6 t^3 + (\chi_8 - 432 \chi_4^2) t^2 - (\chi_{10} - 432 \chi_4 \chi_6) t \\ &\quad \quad \quad + \chi_{12} + 360 \chi_4 \chi_8 - 216 \chi_6^2 - 235008 \chi_4^3 - 648 E_4 \chi_4^2 \Big) x. \end{align*}

\subsection{Rank 13}

The generic transcendental lattice is $$L = L_5 := H \oplus H \oplus A_1(-1)^{\oplus 5}.$$ The general K3 surface in the family --according to Proposition~\ref{prop:weierstrass_13}-- is described by the equation \begin{equation} \label{eqn:rank13}y^2 = x^3 + A(t) x^2 + B(t)x \end{equation} where $A(t) = t^3 + a_4 t + a_6$ and $$B(t) = b_2 t^5 + b_4 t^4 + b_6 t^3 + b_8 t^2 + b_{10} t + b_{12}.$$

Up to the action of the discriminant kernel of $L$, there are six classes of Heegner divisors of discriminant $1/4$:  they are $r_i^{\perp}, \quad i=1,...,5$ and $r^{\perp}$, where $r_i = (0, ..., 1/2, ..., 0) \in A_1^{\oplus 5}$ is nonzero in the $i$th component, and where $r$ can be taken to be the vector \begin{equation} \label{eqn:rank13_r} r = (0, 0) \oplus (1, 1) \oplus (1/2, 1/2, 1/2, 1/2, 1/2) \in H \oplus H \oplus A_1(-1)^{\oplus 5}.\end{equation}
Along the hyperplane $r^{\perp}$, the polarizing lattice extends to $H \oplus D_8(-1) \oplus D_4(-1)$ and the description (\ref{eqn:rank13}) is no longer valid. In concrete terms, this means that $a_k, b_k$ define meromorphic modular forms of level $\mathrm{O}^+(L)$ that are allowed to have poles along the rational quadratic divisor $r^{\perp}$. \\

By Theorem \ref{thm:B5}, the ring $$M^!_*(\mathrm{O}^+(L)) = \Big\{ \text{meromorphic modular forms, holomorphic away from} \; r^{\perp} \Big\}$$ is a free algebra with generators $F_4, E_6$ and $\chi_{2 + 2k}$, $k=0,...,5$, where $E_6$ is an Eisenstein series, $F_4$ is the (unique) holomorphic modular form of weight $4$ with constant Fourier coefficient $1$, and $\chi_{2 + 2k}$ is the (meromorphic) Gritsenko lift of $$\phi_{2+2k}(\tau, z_1,...,z_5) = \Delta(\tau) \prod_{i=1}^5 \phi_{-2, 1}(\tau, z_i) \cdot \sigma_k \Big( \frac{3}{\pi^2} \wp(\tau, z_1),..., \frac{3}{\pi^2} \wp(\tau, z_5) \Big).$$

Furthermore, it follows from \cite{WW23} that the lowest-weight form $\chi_2$ is also a Borcherds product, with divisor $$\mathrm{div}(\chi_2) = 2 \sum_{i=1}^5 r_i^{\perp} - 2r^{\perp}.$$

Under the restriction map to any of the hyperplanes $r_i^{\perp}$ (i.e. to a sublattice of type $H \oplus H \oplus A_1(-1)^{\oplus 4})$, the generating forms are mapped as follows: $$F_4 \mapsto E_4 + 48 \chi_4^{L_4}, \quad E_6^{L_5} \mapsto E_6^{L_4}, \quad \chi_2^{L_5} \mapsto 0,$$ $$\chi_4^{L_5} \mapsto 12 \chi_4^{L_4}, \; \chi_6^{L_5} \mapsto 12 \chi_6^{L_4}, \; \chi_8^{L_5} \mapsto 12 \chi_8^{L_4}, \; \chi_{10}^{L_5} \mapsto 12 \chi_{10}^{L_4}, \; \chi_{12}^{L_5} \mapsto 12 \chi_{12}^{L_4}.$$

There are again Jacobi forms $f_1,...,f_5$ whose $q$-expansions begin $$f_j(\tau, z_1,...,z_5) = \zeta_j^{-1} - 2 + \zeta_j + O(q).$$ Unlike the previous subsections, $f_j$ are no longer uniquely determined; however, there are unique choices of $f_j$ with the property that all of their singular coefficients are represented in the $q^0$-term of its Fourier expansion. In other words, $f_j$ is unique if we ask for $\mathrm{G}(f_j)$ to be holomorphic along the divisor $r^{\perp}$. \\

On the other hand, the polynomial $B(t)$ factors as $$B(t) = (t - \beta_1) \cdot ... \cdot (t - \beta_5),$$ where each $\beta_i$ has weight two under the discriminant kernel and (after a reordering) has double poles exactly on the orbit of $r_i^{\perp}$ and possibly on $r^{\perp}$. Since there are no nonzero, holomorphic modular forms of weight two, we obtain $$\beta_i = B \cdot \chi_2 + C \cdot \mathrm{G}(f_i)$$ for common constants $B, C$. Setting $\chi_2 = 0$, the result of the previous subsection implies $C = 12$ but does not determine $B$. Besides this, the ring structure implies that there are constants $A, C, D, E$ such that \begin{equation} \label{eqn:rank13a4a6} a_4 = -3 F_4 + A \cdot \chi_2^2 \; \text{and} \; a_6 = 2 E_6 + 3 \chi_6 + C \cdot F_4 \chi_2 + D \cdot \chi_2^3 + E \cdot \chi_2 \chi_4. \end{equation}

We will use paramodular restriction as in the previous sections, specializing to $\beta_1 = \beta_2 = \beta_3 = \beta_4 = \beta_5 =: \beta$, to determine the undetermined constants. Under this specialization, the generic transcendental lattice becomes $H \oplus H \oplus A_1(-5)$ and the coefficients become (meromorphic) paramodular forms of level $5$.
These are considerably more complicated than paramodular forms of level $\le 4$; their structure was nevertheless worked out completely in \cite{Wil20}.
The discriminant of $$x^3 + (t^3 + a_4 t + a_6) x^2 + b_2 (t - \beta)^5 x$$ factors as $b_2 (t - \beta)^{10} \cdot P(t)$ with an irreducible polynomial $P$ of degree $6$, whose discriminant itself factors as $$\mathrm{disc}(P) = b_2^3 (\beta^3 + a_4 \beta + a_6)^5 \cdot Q$$ with a factor $Q$ of weight $24$.
In particular, $f := \beta^3 + a_4 \beta + a_6$ is a reflective paramodular form; more precisely, it is a reflective Borcherds product of weight $6$, with a pole of order six (coming from $\beta^3$) on the paramodular Humbert surface $H(1/20)$ that is nonvanishing on the paramodular Humbert surface $H(1)$, and this information determines it uniquely: we have $$\mathrm{div}(f) = -8 H(1/20) + 2 H(1/5) + 4 H(1/4),$$ where $H(D)$ is the (imprimitive) paramodular Heegner divisor of discriminant $D$. In the notation of Table 1 of \cite{Wil20}, we have $$f = \text{const} \cdot (b_8 / b_5)^2.$$
Also, the paramodular restriction $$\chi := \mathrm{res}(\chi_2)$$ of the product $\chi_2$ is the reflective form with divisor $$\mathrm{div} \, \chi = 10 H(1/20) - 2 H(1/4),$$ in the notation of \cite{Wil20}, $$\chi = \text{const} \cdot b_5^2 / b_8.$$

The paramodular restrictions of the forms $\chi_4$ and $\chi_6$ can be computed to be $$\mathrm{res}\, \chi_4 = 30 \chi \cdot \Big( \chi + 2 G \Big)$$ and $$\mathrm{res}\, \chi_6 = 240 \chi \cdot \Big( 5 \chi^2 + 15 \chi G + 6 G^2 \Big),$$
respectively. Inserting Equation (\ref{eqn:rank13a4a6}), we obtain \begin{align*} f &= \beta^3 + a_4 \beta + a_6 \\ &= (B^3 + AB + D + 30E + 3600) \cdot \chi^3 + (36B^2 + 12A + 60E + 10800) G \chi^2 \\ &+ (432B + 4320)G^2 \chi + (C - 3B) F_4 \chi + 1728 G^3 - 36 F_4G + 2E_6. \end{align*}

The paramodular forms $\chi, E_6, F_4, G$ are algebraically independent, and comparing a few Fourier coefficients shows that the unique expression for $f$ in terms of them is $$f = 10 \chi (432 G^2 - 3 F) + 1728 G^3 - 36 F G + 2 E_6.$$
This allows us to read off the missing coefficients $$A = 0, \; B = 0, \; C = -30, \; D = 1800, \; E = -180,$$ and therefore we obtain \begin{align} \label{eqn:rank13_A} a_4 &= -3 F_4, \\ a_6 &= 2 E_6 + 3 \chi_6 - 30 \chi_2 F_4 + 1800 \chi_2^3 - 180 \chi_2 \chi_4,  \end{align} and $$\beta_i = 12 \cdot \mathrm{G}(f_i).$$

The coefficients of $B(t)$ can then be computed as the symmetric polynomials in $\beta_i$. We find \begin{align*} \label{eqn:rank13_B} b_2 &= -12 \chi_2;  \\ b_4 &= -b_2 \sum_{i=1}^5 \beta_i = 12 \chi_4 - 360 \chi_2^2, \\ b_6 &= b_2 \sum_{i < j} \beta_i \beta_j = -12 \chi_6 + 720 \chi_2 \chi_4 - 7200 \chi_2^3, \\ b_8 &= 12 \chi_8 - 432 \chi_4^2 - 14400 \chi_2^2 \chi_4 + 216000 \chi_2^4 - 4320 \chi_2^2 F_4; \\ b_{10} &= -12 \chi_{10} - 936 \chi_2 \chi_8 + 432 \chi_4 \chi_6 + 6480 \chi_2^2 E_6  + 32760 \chi_2^2 \chi_6 + 2376 \chi_2 \chi_4 F_4 \\ &\quad + 27936 \chi_2 \chi_4^2 + 362880 \chi_2^3 F_4 + 21600 \chi_2^3 \chi_4 - 11707200 \chi_2^5; \end{align*}
\begin{align*} b_{12} &= 12 \chi_{12} - 216\chi_2 \chi_{10} + 360\chi_4\chi_8 - 216 \chi_6^2 - 1296 \chi_2 \chi_4 E_6 + 864\chi_2 \chi_6 F_4 - 648 \chi_4^2 F_4 \\ &\quad - 67968 \chi_2^2 \chi_8 + 28584 \chi_2 \chi_4 \chi_6 - 16992 \chi_4^3 + 83520\chi_2^3 E_6 - 5832 \chi_2^2 F_4^2  - 258336 \chi_2^2 \chi_4 F_4 \\ & \quad - 574560 \chi_2^3 \chi_6 + 871776 \chi_2^2 \chi_4^2 + 25574400 \chi_2^4 F_4 + 132096960 \chi_2^4 \chi_4 - 1523059200 \chi_2^6. \end{align*}

\begin{rem} A priori, the coefficients $a_k$ and $b_k$ were allowed to have poles along the exceptional hyperplane $r^{\perp}$ (of order at most their weight). However, the computation of the coefficients shows that when we write $$y^2 = x^3 + (t^3 + a_4 t + a_6) x^2 + (b_2 t^5 + ... + b_{12}) x,$$ $a_4$ and $a_6$ are holomorphic along $r^{\perp}$ and all the $b_k$ have exactly a double pole there.
\end{rem}

\section{The remaining polarization}

We have now computed the coefficients of van Geemen--Sarti equations for all of the lattice polarizations of Equation (\ref{eqn:lattices}) except for $S = H \oplus D_8(-1) \oplus D_4(-1)$. The orthogonal complement in the K3 lattice is $L = H \oplus H(2) \oplus D_4(-1)$, and it embeds (up to an isometry) into the lattice $L_5 = H \oplus H \oplus A_1(-1)^{\oplus 5}$ as the exceptional hyperplane $r^{\perp}$ defined in (\ref{eqn:rank13_r}).

\subsection{Modular forms and a Weierstrass equation for \texorpdfstring{$S = H \oplus D_8(-1) \oplus D_4(-1)$}{S}}

The graded algebra $M_*(\mathrm{O}^+(L))$ of modular forms was described in \cite{WW23b}. It is a free algebra that is generated in weights $2, 6, 8, 10, 12, 16, 20$. The generators of weight $2$ and $6$ can be chosen to be the (unique) modular forms $E_2$ and $F_6$ whose leading Fourier-Jacobi coefficients are $$1 + 24q + 24q^2 + 96q^3 + ... \in M_2(\Gamma_0(2))$$ and $$q + 32q^2 + 244q^3 + 1024q^4 + ... \in M_6(\Gamma_0(2)),$$ and the generator $F_{10}$ of weight $10$ will be chosen to be the (unique) cusp form of that weight, normalized to have coprime Fourier coefficients.
The other generators can be described as follows: there are modular forms $g_1,...,g_5$ of weight four for the discriminant kernel $\mathrm{\widetilde{O}}(L)$ that are permuted under the action of $\mathrm{O}^+(L)$ and which satisfy $$E_2^2 = \sum_{i=1}^5 g_i.$$ The generators of weights $8, 12, 16, 20$ can be taken as the power sums $$p_{4k} = \sum_{i=1}^5 g_i^k, \quad k \in \{2,3,4,5\}.$$

Following Proposition \ref{prop:weierstrass}, the van Geemen--Sarti equation for the general $S$-polarized K3 surface can be put in the form $$y^2 = x^3 + 2 C(t) x^2 + D(t) x,$$ where $$C(t) = c_2 t^2 + c_6 t + c_{10}$$ and $$D(t) = t^5 + d_8 t^3 + d_{12} t^2 + d_{16} t + d_{20},$$ and where the $c_k$ and $d_k$ are modular forms for $\mathrm{O}^+(L)$ of weight $k$.
If we factor $$D(t) = \prod_{i=1}^5 (t - \gamma_i),$$ then each $\gamma_i$ is modular of weight $4$ on the discriminant kernel. Since they sum to zero, the ring structure implies that $\gamma_i$ is a constant multiple of $g_i - \frac{1}{5} E_2^2$: without loss of generality (applying a substitution if necessary), we may assume the roots are $$\gamma_i = g_i - \frac{1}{5} E_2^2, \quad i=1,2,3,4,5.$$

\subsection{A limit process}

Consider the equation $$y^2 = x^3 + (t^3 + a_4 t + a_6) x^2 + b_2 x \cdot \prod_{t=1}^5 (t - \beta_i)$$ with $a_4, a_6, b_2$ and $\beta_i$ defined as in Section 5.6.
The forms $a_4, a_6$ and $\beta_i$ are holomorphic along $r^{\perp}$ and their restrictions can be determined using the ``pullback trick", namely $$\mathrm{res}(\beta_i) = E_2, \quad i=1,2,3,4,5$$ and $$\mathrm{res}(a_4) = -3E_2^2, \quad \mathrm{res}(a_6) = 2 E_2^3.$$

Define the following (holomorphic) modular forms of weight 4 and level $\mathrm{\widetilde{O}}(L_5)$: $$f_j := b_2 \cdot \Big(-\beta_j + \frac{1}{5} \sum_{i=1}^5 \beta_i \Big) = \sum_{i \ne j} \frac{b_2 (\beta_i - \beta_j)}{5}, \quad j=1,...,5.$$
These are Gritsenko lifts; they can be written as the additive theta lifts of the unique vector-valued modular forms $\varphi_j$ of weight $3/2$ on the Weil representation attached to $L_5$ with constant Fourier coefficient $$\frac{576}{5} \mathfrak{e}_{\gamma_i} - \frac{144}{5} \sum_{j \ne i} \mathfrak{e}_{\gamma_j},$$ where $\gamma_1,...,\gamma_5$ are the nonzero cosets of $(L_5' / L_5)$ of norm zero, and $\mathfrak{e}_{\gamma}$ are the associated basis elements of the group ring $\mathbb{C}[(L_5' / L_5)]$. (For the definition of the Weil representation we refer to \cite{B98}.) This allows the restrictions of $f_j$ to $r^{\perp}$ (which are nonzero) to be computed using the pullback method as well: $$\mathrm{res}(f_j) = -3 \cdot g_j.$$

In particular, the forms $h_j := \frac{C}{3} (a_4 / 5 - f_j)$ restrict to the zeros $\gamma_j$ of $D$.

This motivates the following substitution: $$t = \frac{1}{b_2} \cdot T + \frac{1}{5} \Big( -\frac{a_4}{b_2} + \sum_{j=1}^5 \beta_j \Big) = \frac{1}{ b_2} \cdot T - \frac{a_4 + b_4}{5 b_2},$$ such that $$t - \beta_j = \frac{1}{b_2} \cdot \Big( T - ( \tfrac{a_4}{5} - f_j) \Big).$$

Together with the substitutions $x = X / b_2^2$ and $y = Y / b_2^3$, this transforms the equation $$y^2 = x^3 + A(t) x^2 + B(t) x$$ into $$Y^2 = X^3 + b_2^2 A \left( \frac{1}{b_2} \cdot T - \frac{a_4 + b_4}{5 b_2} \right) X^2 + \prod_{j=1}^5 (T - h_j)  \cdot X.$$

Here, 
\begin{equation} 
\begin{split}
\label{eqn:transformed_rank_13} 
b_2^2 A \left( \frac{1}{b_2} \cdot T - \frac{a_4 + b_4}{5b_2} \right) & = \frac{1}{b_2} T^3 - \frac{3}{5} \frac{a_4 + b_4}{b_2} T^2 + \left( b_2 a_4 + \frac{3}{25} \frac{a_4^2}{b_2} + \frac{6}{25} \frac{a_4 b_4}{b_2} + \frac{3}{25} \frac{b_4^2}{b_2} \right) T 
\\ &- \frac{1}{5} b_2 a_4^2 - \frac{1}{5} b_2 a_4 b_4 + b_2^2 a_6 - \frac{1}{125} \frac{a_4^3}{b_2} - \frac{3}{125} \frac{a_4^2 b_4}{b_2} - \frac{3}{125} \frac{a_4 b_4^2}{b_2} - \frac{b_4^3}{125 b_2}. 
\end{split}
\end{equation}

The asymptotic behavior of the coefficients $b_2, b_4$ in the limit is governed by Theorem \ref{th:gritsenko} and Equation (\ref{eqn:ppart}): letting $r^{\perp}$ be the exceptional hyperplane  we obtain, asymptotically as $\langle r, Z \rangle \rightarrow 0,$ $$b_2(Z) \sim -12 \cdot -\frac{1}{4\pi^2 \langle r, Z \rangle^2} + E_2 + O(\langle r, Z \rangle)$$ and $$b_4(Z) = -b_2(Z) \sum_{i=1}^5 \beta_i(Z) \sim 12 \cdot -\frac{5 E_2}{4\pi^2 \langle r, Z \rangle^2}.$$
This shows that $$3 b_2 a_4 + \frac{9}{25} \frac{b_4^2}{b_2} \sim -\frac{9}{\pi^2 \langle r, Z \rangle^2} E_2^2 + \frac{9}{\pi^2 \langle r, Z \rangle^2} E_2^2,$$ 
and similarly that $$-\frac{1}{5}b_2 a_4^2 - \frac{1}{5} b_2 a_4 b_4 + b_2^2 a_6 - \frac{3}{125} \frac{a_4 b_4^2}{b_2} - \frac{b_4^3}{125 b_2}$$ remains holomorphic in the limit $1/b_2 \rightarrow 0$.

\subsection{Determining the coefficients}

The coefficients $a_k, b_k$ of the van Geemen--Sarti equation $$y^2 = (t^3 + a_4 t + a_6) x^2 + (b_2 t^5 + ... + b_{12}) x$$ have ``level two" Fourier-Jacobi expansions in addition to level one Fourier-Jacobi expansions, due to the isometry $$L_5 = H \oplus H \oplus A_1(-1)^{\oplus 5} = H \oplus H(2) \oplus D_4(-1) \oplus A_1(-1).$$
More precisely, we can represent modular forms on $\mathrm{O}^+(L_5)$ as series $$F(Z) = \sum_{n=0}^{\infty} \phi_n(\tau, z) e^{2\pi i n w}$$ where each $\phi_n$ is a meromorphic Jacobi form of lattice index $D_4 \oplus A_1$ with respect to the congruence subgroup $\Gamma_0(2) \le \mathrm{SL}_2(\mathbb{Z})$. The two Fourier-Jacobi expansions can be quite different, and it is not generally clear how to compute the level two expansions of a modular form that is given as a level one Fourier-Jacobi series, but the expansions of a Gritsenko lift or Borcherds product at either cusp can be computed using the formulas of \cite{B98}.

In the level two expansions $$\chi_k = \sum_{n=0}^{\infty} \chi_{k, n}(\tau, \mathfrak{z}, z) e^{2\pi i n w}, \quad \mathfrak{z} \in D_4 \otimes \mathbb{C}, \; z \in A_1 \otimes \mathbb{C},$$ the terms $\chi_{k, 0}$ are independent of $\mathfrak{z}$ and define elliptic functions (in the variable $z$) on the lattice $\mathbb{Z} \oplus 2\tau \mathbb{Z}$ which are holomorphic away from the lattice points. For example, we have \begin{align*} \chi_{2, 0} &= \frac{1}{(2\pi i)^2} \wp(2\tau, z) + \frac{1}{12} E_2(\tau) - \frac{1}{6} E_2(2\tau) \\ &= \frac{1}{\zeta^{-1} - 2 + \zeta} - 2q + (\zeta^{-1} - 4 + \zeta) q^2 - 8q^3 + (2 \zeta^{-2} + \zeta^{-1} - 8 + \zeta + 2 \zeta^2) q^4 + O(q^5),\end{align*}
and
\begin{align*} \chi_{4, 0} &= \frac{5}{(2\pi i)^4} \wp''(2\tau, z) + 2\frac{ \eta(2\tau)^{16}}{\eta(\tau)^8} \\ &= \frac{5 \zeta^{-1} + 20 + 5 \zeta}{(\zeta^{-1} - 2 + \zeta)^2} + 2q + (5 \zeta^{-1} + 16 + 5\zeta) q^2 + 56q^3 \\
& \qquad \qquad + (40 \zeta^{-2} + 5 \zeta^{-1} + 128 + 5 \zeta + 40 \zeta^2) q^4 + O(q^5), \end{align*} and the higher terms $\chi_{k, n},$ $n > 0$ are weak Jacobi forms of lattice index.

Setting $z=0$, or equivalently $\zeta = 1$ in the Fourier expansions of $\chi_k$ and using the formulas from before yields $$\mathrm{res} \Big( -\frac{3}{5} \frac{a_4 + b_4}{b_2} \Big) = 3 E_2;$$ $$\mathrm{res} \Big( b_2 a_4 + \frac{3}{25} \frac{a_4^2}{b_2} + \frac{6}{25} \frac{a_4 b_4}{b_2} + \frac{3}{25} \frac{b_4^2}{b_2} \Big) = \frac{18}{5} E_2^3 - \frac{864}{5} F_6;$$ and \begin{align*} &\quad \mathrm{res} \Big( -\frac{1}{5} b_2 a_4^2 - \frac{1}{5} b_2 a_4 b_4 + b_2^2 a_6 - \frac{1}{125} \frac{a_4^3}{b_2} - \frac{3}{125} \frac{a_4^2 b_4}{b_2} - \frac{3}{125} \frac{a_4 b_4^2}{b_2} - \frac{b_4^3}{125 b_2}\Big) \\ &= \frac{81}{100} E_2^5 - \frac{2592}{25} E_2^2 F_6 + \frac{27}{20} E_2 p_8 + \frac{31104}{5} F_{10}. \end{align*} So in the limit $b_2 \rightarrow \infty$, Equation (\ref{eqn:transformed_rank_13}) becomes $$Y^2 = X^3 + 2 C(T) X^2 + D(T) X$$ with $$C(T) = \frac{3}{2} E_2 \cdot T^2 + \Big( \frac{9}{5} E_2^3 - \frac{432}{5} F_6 \Big) T + \frac{81}{200} E_2^5 - \frac{1296}{25} E_2^2 F_6 + \frac{27}{40} E_2 p_8 + \frac{15552}{5} F_{10}.$$
The coefficients of $$D(T) = \prod_{i=1}^5 (T - \gamma_i) = T^5 + d_8 T^3 + d_{12} T^2 + d_{16} T + d_{20}$$ are easily determined in terms of $E_2$ and $p_{4k} = \sum_i g_i^k$ by expanding the product and substituting $\gamma_i = g_i - \frac{1}{5} E_2^2 = g_i - \frac{1}{5} \sum_{j=1}^5 g_j$. We obtain $$d_8 = -\frac{1}{2} \sum_{i=1}^5 \gamma_i^2 = -\frac{1}{2} p_8 + \frac{1}{10} E_2^4;$$ $$d_{12} = -\frac{1}{3} \sum_{i=1}^5 \gamma_i^3 = -\frac{1}{3} p_{12} + \frac{1}{5} E_2^2 p_8 - \frac{2}{75} E_2^6;$$ $$d_{16} = -\frac{1}{4} \sum_{i=1}^5 \gamma_i^4 + \frac{1}{8} \Big( \sum_{i=1}^5 \gamma_i^2 \Big)^2 = -\frac{1}{4} p_{16} + \frac{1}{5} E_2^2 p_{12} - \frac{11}{100} E_2^4 p_8 + \frac{11}{1000} E_2^8 + \frac{1}{8} p_8^2;$$ \begin{align*} d_{20} &= -\frac{1}{5} \sum_{i=1}^5 \gamma_i^5 + \frac{1}{6} \Big( \sum_{i=1}^5 \gamma_i^3 \Big) \Big( \sum_{i=1}^5 \gamma_i^2 \Big) \\ &= -\frac{1}{5} p_{20} + \frac{1}{5} E_2^2 p_{16} - \frac{17}{150} E_2^4 p_{12} + \frac{37}{750} E_2^6 p_8 - \frac{37}{9375} E_2^{10} + \frac{1}{6} p_{12} p_8 - 
\frac{1}{10} E_2^2 p_8^2 \end{align*}

\appendix
\section{Appendix}
\label{appendix}

In this appendix, we will describe the coefficients of the van Geemen--Sarti equations determined in Sections 5 and 6.
\smallskip

In the first five tables, $E_4$ and $E_6$ denote the normalized (constant term 1) Eisenstein series of weights $4$ and $6$, and $\chi_n,...,\chi_{12}$ are the Gritsenko lifts

$$\chi_{12 - 2n + 2k} = \mathrm{G}(\phi_{12 - 2n + 2k}),$$ where $\phi_{12 - 2n + 2k}$ is the Jacobi form $$\phi_{12 - 2n + 2k}(\tau, z_1,...,z_n) = \Delta(\tau) \Big( \prod_{j=1}^n \phi_{-2, 1}(\tau, z_j) \Big) \sum_{i_1 < ... < i_k} \frac{\phi_{0, 1}(\tau, z_{i_1}) \cdot ... \cdot \phi_{0, 1}(\tau, z_{i_k})}{\phi_{-2, 1}(\tau, z_{i_1}) \cdot ... \cdot \phi_{-2, 1}(\tau, z_{i_k})}$$ of weight $12-2n+2k$.

\begin{table}[htbp]
\centering
\caption{Rank 18. \\ $y^2 = x^3 + (t^3 + a_4 t + a_6) x^2 + b_{12} x$}
\begin{tabular}{c|c}
\hline
$a_4$ & $-3 E_4$ \\
\hline
$a_6$ & $2 E_6$ \\
\hline
$b_{12}$ & $12^6 \chi_{12}$\\
\hline
\end{tabular}
\end{table}

\begin{table}[htbp]
\centering
\caption{Rank 17. \\ $y^2 = x^3 + (t^3 + a_4 t + a_6) x^2 + (b_{10}t + b_{12}) x$}
\begin{tabular}{c|c}
\hline
$a_4$ & $-3 E_4$ \\
\hline
$a_6$ & $2 E_6$ \\
\hline
$b_{10}$ & $-12^5 \chi_{10}$ \\
\hline
$b_{12}$ & $12^5 \chi_{12}$\\
\hline
\end{tabular}
\end{table}

\begin{table}[htbp]
\centering
\caption{Rank 16. \\ $y^2 = x^3 + (t^3 + a_4 t + a_6) x^2 + (b_8 t^2 + b_{10}t + b_{12}) x$}
\begin{tabular}{c|c}
\hline
$a_4$ & $-3 E_4$ \\
\hline
$a_6$ & $2 E_6$ \\
\hline
$b_8$ & $12^4 \chi_8$ \\
\hline
$b_{10}$ & $-12^4 \chi_{10}$ \\
\hline
$b_{12}$ & $12^4 \chi_{12}$\\
\hline
\end{tabular}
\end{table}

\begin{table}[htbp]
\centering
\caption{Rank 15. \\ $y^2 = x^3 + (t^3 + a_4 t + a_6) x^2 + (b_6 t^3 + b_8 t^2 + b_{10}t + b_{12}) x$}
\begin{tabular}{c|c}
\hline
$a_4$ & $-3 E_4$ \\
\hline
$a_6$ & $2 E_6 + \frac{2^7 \cdot 3^4}{17} \chi_6$ \\
\hline
$b_6$ & $-12^3 \chi_6$ \\
\hline
$b_8$ & $12^3 \chi_8$ \\
\hline
$b_{10}$ & $-12^3 \chi_{10}$ \\
\hline
$b_{12}$ & $12^3 \chi_{12} - 2^{11} \cdot 3^7 \chi_6^2$\\
\hline
\end{tabular}
\end{table}

\begin{table}[htbp]
\centering
\caption{Rank 14. \\ $y^2 = x^3 + (t^3 + a_4 t + a_6) x^2 + (b_4 t^4 + b_6 t^3 + b_8 t^2 + b_{10}t + b_{12}) x$}
\begin{tabular}{c|c}
\hline
$a_4$ & $-3 E_4 - 144 \chi_4$ \\
\hline
$a_6$ & $2 E_6 + 432 \chi_6$ \\
\hline
$b_4$ & $144 \chi_4$ \\
\hline
$b_6$ & $-144 \chi_6$ \\
\hline
$b_8$ & $144 \chi_8 - 2^8 \cdot 3^5 \chi_4^2$ \\
\hline
$b_{10}$ & $-144 \chi_{10} + 2^8 \cdot 3^5 \chi_4 \chi_6$ \\
\hline
$b_{12}$ & $144 \chi_{12} + 2^7 3^4 5^1 \cdot \chi_4 \chi_8 - 2^7 3^5 \chi_6^2 - 2^{13} 3^5 \cdot 17 \chi_4^3 - 2^7 3^6 E_4 \chi_4^2$\\
\hline
\end{tabular}
\end{table}

The notation is almost the same in the following table, but the Eisenstein series $E_4$ fails to be holomorphic. Instead $F_4$ denotes the (unique) modular form in $M_4(\mathrm{O}^+(L))$, $L = H \oplus H \oplus A_1(-1)^{\oplus 5}$, normalized to have constant coefficient $1$.

\begin{table}[htbp]
\centering
\caption{Rank 13. \\ $y^2 = x^3 + (t^3 + a_4 t + a_6) x^2 + (b_2 t^5 + b_4 t^4 + b_6 t^3 + b_8 t^2 + b_{10}t + b_{12}) x$}
\begin{tabular}{c|c}
\hline
$a_4$ & $-3 F_4$ \\
\hline
$a_6$ & $2 E_6 + 3\chi_6 - 30 \chi_2 F_4 - 180 \chi_2 \chi_4 + 1800 \chi_2^3$ \\
\hline
$b_2$ & $-12 \chi_2$ \\
\hline
$b_4$ & $12 \chi_4 - 360 \chi_2^2$ \\
\hline
$b_6$ & $-12 \chi_6 + 720 \chi_2 \chi_4 - 7200 \chi_2^3$ \\
\hline
$b_8$ & $12 \chi_8 - 432 \chi_4^2 - 14400 \chi_2^2 \chi_4 - 4320 \chi_2^2 F_4 + 216000 \chi_2^4$ \\
\hline
$b_{10}$ &  $-12 \chi_{10} - 936 \chi_2 \chi_8 + 432 \chi_4 \chi_6 + 6480 \chi_2^2 E_6  + 32760 \chi_2^2 \chi_6 + 2376 \chi_2 \chi_4 F_4$ \\ & $+ 27936 \chi_2 \chi_4^2 + 362880 \chi_2^3 F_4 + 21600 \chi_2^3 \chi_4 - 11707200 \chi_2^5 $\\
\hline
$b_{12}$ & $12 \chi_{12} - 216\chi_2 \chi_{10} + 360\chi_4\chi_8 - 216 \chi_6^2 - 1296 \chi_2 \chi_4 E_6 + 864\chi_2 \chi_6 F_4 - 648 \chi_4^2 F_4$ \\ &  $- 67968 \chi_2^2 \chi_8 + 28584 \chi_2 \chi_4 \chi_6 - 16992 \chi_4^3 + 83520\chi_2^3 E_6 - 5832 \chi_2^2 F_4^2  - 258336 \chi_2^2 \chi_4 F_4$ \\ & $- 574560 \chi_2^3 \chi_6 + 871776 \chi_2^2 \chi_4^2 + 25574400 \chi_2^4 F_4 + 132096960 \chi_2^4 \chi_4 - 1523059200 \chi_2^6$  \\
\hline
\end{tabular}
\end{table}

In the following table, we have $L = H \oplus H(2) \oplus D_4(-1)$; $E_2$ is the unique modular form of weight $2$ whose leading Fourier--Jacobi coefficient at the level two cusp is $$1 + 24q + 24q^2 + 96q^3 + ...;$$ $F_6$ is the unique modular form of weight $6$ whose leading Fourier--Jacobi coefficient at the level two cusp is $$q + 32q^2 + 244q^3 + ...;$$ and $F_{10}$ is the unique cusp form of weight $10$, normalized to have coprime integral Fourier coefficients. \\
The forms $p_8, p_{12}, p_{16}$ and $p_{20}$ are defined as the power sums $$p_{4k} = \sum_{i=1}^5 g_i^k,$$ where $g_1,...,g_5$ are modular forms of weight $4$ with respect to the discriminant kernel of $L$ which satisfy $\sum_{i=1}^5 g_i = E_2^2.$

\begin{table}[htbp]
\centering
\caption{Rank 14. \\ $y^2 = x^3 + 2(c_2 t^2 + c_6 t + c_{10}) x^2 + (t^5 + d_8 t^3 + d_{12} t^2 + d_{16} t + d_{20})x$}
\begin{tabular}{c|c}
\hline
$c_2$ & $\frac{3}{2} E_2$ \\
\hline
$c_6$ & $\frac{9}{5}E_2^3 - \frac{432}{5} F_6$ \\
\hline
$c_{10}$ & $\frac{81}{200} E_2^5 - \frac{1296}{25} E_2^2 F_6 + \frac{27}{40} E_2 p_8 + \frac{15552}{5} F_{10}$ \\
\hline
$d_8$ & $-\frac{1}{2} p_8 + \frac{1}{10} E_2^4$ \\
\hline
$d_{12}$ & $-\frac{1}{3} p_{12} + \frac{1}{5} E_2^2 p_8 - \frac{2}{75} E_2^6$ \\
\hline
$d_{16}$ & $-\frac{1}{4}p_{16} + \frac{1}{8} p_8^2 + \frac{1}{5} E_2^2 p_{12} - \frac{11}{100} E_2^4 p_8 + \frac{11}{1000} E_2^8$ \\
\hline
$d_{20}$ & $-\frac{1}{5} p_{20} + \frac{1}{6} p_{12} p_8 + \frac{1}{5} E_2^2 p_{16} - \frac{1}{10} E_2^2 p_8^2 - \frac{17}{150} E_2^4 p_{12} + \frac{37}{750} E_2^6 p_8 - \frac{37}{9375} E_2^{10}$ \\
\hline
\end{tabular}
\end{table}

\bibliographystyle{amsplain}
\bibliography{_bib.general}{}
\end{document}